\documentclass[nosumlimits,twoside]{amsart}
\usepackage{eurosym}
\usepackage{amsfonts, amsmath, amssymb}
\usepackage[bookmarksnumbered,plainpages,driverfallback=hypertex]{hyperref} 
\usepackage{graphicx}
\usepackage{float}
\usepackage{srcltx}
\usepackage[all]{xy}
\usepackage{version}
\usepackage[final]{showlabels}
\usepackage[T1]{fontenc}
\usepackage{xcolor}
\usepackage{enumerate}
\usepackage[normalem]{ulem}
\usepackage{bm}
\usepackage{latexsym}
\usepackage[2emode]{psfrag}
\usepackage{yhmath}
\usepackage{array}
\usepackage{dsfont}

\newtheorem{theorem}{\sc Theorem}[section]
\newtheorem{proposition}[theorem]{\sc Proposition}

\newtheorem{lemma}[theorem]{\sc Lemma}
\newtheorem{corollary}[theorem]{\sc Corollary}
\theoremstyle{definition}
\newtheorem{definition}[theorem]{\sc Definition}

\newtheorem{example}[theorem]{\sc Example}

\theoremstyle{remark}
\newtheorem{remark}[theorem]{\sc Remark}

\newtheorem{notations}[theorem]{\sc Setting}

\newenvironment{invisible}{{\noindent\sc \colorbox{yellow}{Invisible:}\;}\color{gray}}{\medskip}

\excludeversion{invisible}
\allowdisplaybreaks
\newcommand{\id}{{\sf id}}
\newcommand{\Id}{{\sf Id}}
\newcommand{\Cc}{\mathcal{C}}
\newcommand{\Dd}{\mathcal{D}}

\newcommand{\Aa}{\mathcal{A}}
\newcommand{\Bb}{\mathcal{B}}
\newcommand{\Ee}{\mathcal{E}}
\newcommand{\Ss}{\mathcal{S}}
\newcommand{\ot}{\otimes}
\def\Vec{{\sf Vec}}
\def\Alg{{\sf Alg}}

\def\Coalg{{\sf Coalg}}

\def\Bialg{{\sf Bialg}}

\def\Vec{{\sf Vec}}
\def\indobj{{\sf IndObj}}
\def\epindobj{{\sf IndObj^e}}
\def\indalg{{\sf IndAlg}}
\def\epindalg{{\sf IndAlg^e}}
\def\cocone{{\sf Cocone}}

\def\uoR{\underline{\overline{R}}}
\def\uoL{\underline{\overline{L}}}

\def\hom{{\rm Hom}}

\def\pulb{\ar@{}[dr]|(0.2){\mbox{\Large{$\lrcorner$}}}}

\newcommand{\ev}{{\rm ev}}

\newcommand{\unit}{\mathds{I}}
\newcommand{\op}{\mathrm{op}}
\newcommand{\diamw}{{\scriptscriptstyle \lozenge}}
\newcommand{\diamb}{{\scriptscriptstyle \blacklozenge}}

\begin{document}
\title[Liftable pairs of functors and Initial objects]{Liftable pairs of functors and Initial objects}

\thanks{This article was written while the first and the third author were
members of the National Group for Algebraic and Geometric Structures, and
their Applications (GNSAGA-INdAM). They were both partially supported by
MIUR within the National Research Project PRIN 2017. The first author was
partially supported by the research grant ``Progetti di Eccellenza
2011/2012'' from the ``Fondazione Cassa di Risparmio di Padova e Rovigo''.
He thanks the members of the department of Mathematics of both Vrije
Universiteit Brussel and Universit\'{e} Libre de Bruxelles for their warm
hospitality and support during his stay in Brussels in August 2013, when the
work on this paper was initiated. The second named author acknowledges the
financial support of an INdAM Marie Curie Fellowship. \\
The authors would also like to thank Joost Vercruysse and Miodrag C. Iovanov
for helpful discussions and the referee for pointing out a mistake in a
previous version of this article.}

\begin{abstract}
Let $\Aa  $ and $\Bb  $ be monoidal categories and let $R:\Aa  \rightarrow \Bb$ be a lax monoidal functor. If $R$ has a left adjoint $L$, it is well-known that the two adjoints induce functors $\overline{R}=\Alg(R):\Alg ({\Aa  })\rightarrow {\Alg }({\Bb  })$ and $\underline{L}=\Coalg(L):{%
\Coalg }({\Bb  })\rightarrow \Coalg ({\Aa  })$ respectively. The pair $(L,R)$ is called \textit{%
liftable} if the functor $\overline{R}$ has a left adjoint and if the
functor $\underline{L}$ has a right adjoint. A pleasing fact is that, when $\Aa  $, $\Bb  $ and $R$ are moreover braided, a liftable pair of functors as above gives rise to
an adjunction at the level of bialgebras.

In this note, sufficient conditions on the category $\Aa  $ for $%
\overline{R}$ to possess a left adjoint, are given. Natively these conditions involve the existence of suitable colimits that we interpret as objects which are simultaneously initial in four distinguished categories (among which the category of epi-induced objects), allowing for an explicit construction of $\overline{L}$, under the appropriate hypotheses. This is achieved by introducing a relative version of the notion of weakly coreflective subcategory, which turns out to be a useful tool to compare the initial objects in the involved categories.

We apply our results to obtain an analogue of Sweedler's finite dual for the
category of vector spaces graded by an abelian group $G$ endowed with a bicharacter. When the bicharacter on $G$ is skew-symmetric, a lifted adjunction as mentioned above is explicitly
described, inducing an auto-adjunction on the category of bialgebras
``colored'' by $G$.
\end{abstract}

\author{Alessandro Ardizzoni}
\address{%
\parbox[b]{\linewidth}{University of Turin, Department of Mathematics ``G. Peano'', via
Carlo Alberto 10, I-10123 Torino, Italy}}
\email{alessandro.ardizzoni@unito.it}
\urladdr{\url{http://sites.google.com/site/aleardizzonihome}}
\author{Isar Goyvaerts}
\address{%
\parbox[b]{\linewidth}{Department of Mathematics, Faculty of
Engineering, Vrije Universiteit Brussel, Pleinlaan 2, B-1050 Brussel,
Belgium}}
\email{Isar.Goyvaerts@vub.be}
\author{Claudia Menini}
\address{%
\parbox[b]{\linewidth}{University of Ferrara, Department of Mathematics and Computer Science, Via Machiavelli
30, Ferrara, I-44121, Italy}}
\email{men@unife.it}
\urladdr{\url{http://sites.google.com/a/unife.it/claudia-menini}}
\maketitle
\tableofcontents

\section*{Introduction}

Let $\Aa  $ and $\Bb  $ be monoidal categories and let $L\dashv R:\Aa  \rightarrow \Bb$ be adjoint functors. It is well-known that, if $L$ can
be endowed with the structure of a colax monoidal functor, then $R$ becomes
a lax monoidal functor, and the other way around. Letting $(R,\phi_2 ,\phi_0
):\Aa  \rightarrow \Bb  $ now be a lax monoidal functor, $R$
induces a functor $\overline{R}=\Alg(R):\Alg (\Aa)\rightarrow
\Alg (\Bb)$ between the respective categories of algebra
objects. Dually, a colax monoidal functor $(L,\psi _{2},\psi _{0}):\mathcal{B%
}\rightarrow \Aa  $ colifts to a functor $\underline{L}=\Coalg(L):\Coalg (\Bb)\rightarrow \Coalg (\Aa)$ between the respective categories of coalgebra objects.
In the article \cite{GV-OnTheDuality}, an adjoint pair of functors $(L,R)$
between monoidal categories $\Aa  $ and $\Bb  $ such that $R$ is
a lax monoidal functor (or, equivalently, $L$ is colax monoidal) is called
\textit{liftable} if the functor $\overline{R}$ has a left adjoint, say $\overline{L}$, and if
the functor $\underline{L}$ has a right adjoint, say $\underline{R}$. If $\Aa  $ and $%
\Bb  $ come both endowed with a braiding, it is shown in loc. cit.
that such a liftable pair of functors $(L,R)$ gives rise to an adjunction
between the respective categories of bialgebra objects
\begin{equation*}
\xymatrixcolsep{1.5cm}\xymatrix{ \Bialg(\Aa) \ar@<.5ex>[rr]^-{\uoR=\Alg(\underline{R})} && \Bialg(\Bb)
\ar@<.5ex>[ll]^-{\uoL=\Coalg(\overline{L})} }
\end{equation*}
provided the functor $R$ enjoys the property of being braided with respect
to the braidings of $\Aa  $ and $\Bb  $ (cf. \cite[Theorem 2.7]%
{GV-OnTheDuality}).

A prototypical example of a liftable pair of functors is obtained by taking $%
\Bb  $ to be the symmetric monoidal category $\Vec$ of $k$-vector
spaces ($k$ a field) where the symmetry is just the twist. Putting $\mathcal{%
A}$ to be the opposite category $\Vec^\op$ of $\Vec  $ and taking the vector space
dual $X^{*}=\hom _{k}(X,k)$, one obtains a (covariant) adjunction $L\dashv R:\Aa\to\Bb$ with $L=(-)^*$ and  $R=(-)^*$.
The functor $R$ satisfies the necessary conditions to induce a functor $%
\overline{R}=\Alg(R):\Alg (\Aa  )\to \Alg (\Bb  )$.
Explicitly, one obtains that $\overline{R}$ is the well-known functor that
computes the dual algebra of a $k$-coalgebra (remark that the functor $%
\underline{L}=\Coalg(L):\Coalg (\Bb  )\to\Coalg (\Aa  )$ is
exactly the same functor). A left adjoint $\overline{L}=(-)^\circ$ for $\overline{R}$ is given by the
functor that assigns the so-called finite dual coalgebra $A^\circ$ to a $k$%
-algebra $A$; this construction is originally due to Sweedler, see \cite%
{Sweedler-Hopf}. Noticing that the very same construction provides a right
adjoint for the functor $\underline{L}$, one obtains that the pair $(L,R)$
is indeed liftable and, applying the above-cited theorem, one recovers the
result that the finite dual induces an auto-adjunction on the category of $k$%
-bialgebras (cf. \cite[page 87]{Abe-Hopf}, for instance).
Generalizations of this construction have
been studied by different authors, see e.g. \cite{AGW, CN, Po2} and \cite{PS}.

Let us go back to our general setting of a functor $R:\Aa\to \Bb$ as at the beginning of this introductory section and notice that being liftable really \textit{is} a condition: there exist examples of lax monoidal functors $R$ between monoidal categories that have
a left adjoint $L$, but for which $\overline{R}$ does not have a left
adjoint (cf. \cite[Example 4.2]{AGM}). One aim of this paper is to give sufficient conditions on the category $\Aa  $ for the functor $\overline{R}=\Alg(R):\Alg(\Aa)\to\Alg(\Bb)$ to possess a left adjoint $\overline{L}$. As we will see, these conditions involve the existence of  suitable colimits that we manage to interpret as objects which are simultaneously initial in four distinguished categories, among them the category of epi-induced objects, which allows for an explicit construction of $\overline{L}$, under the appropriate hypotheses. This is performed by introducing a relative version of the notion of weakly coreflective subcategory that allows, among other things, to identify the initial objects in these categories and that we find is of independent interest.

In the article \cite{AGM}, a context, appeared in its original form in \cite{GV-OnTheDuality}, where the liftability assumption can be proved to hold is studied: a so-called \textit{pre-rigid} braided
monoidal category $\Cc $ always allows for a liftable pair of adjoint
functors $L\dashv R:\Cc ^\op\rightarrow \Cc$, with $L=(-)^{*}$ and $R=(-)^{*}$,
 provided $\overline{R}$ has a left adjoint. In the present paper, we consider $\Cc $
to be the category of vector spaces graded by an abelian group $G$. When
being given a skew-symmetric bicharacter on $G$, the lifted adjunction between bialgebras
in $\Cc $ can be explicitly computed and provides a $G$-graded
version of Sweedler's classical finite dual construction. To the best of our
knowledge, this application does not appear elsewhere in literature. Let us
sketch in more detail how we go about this computation. \medskip

In Section \ref{preliminares}, we start by recalling some of the notions we use in the paper, among them the one of liftable pair of adjoint functors and its behaviour on braided categories. Then, in Subsection  \ref{tamDub}, we present sufficient conditions for $\overline{R}$
to possess a left adjoint, provided some extra conditions on the category $%
\Aa  $ hold (cf. Theorem \ref{teo:Tambara}). This is obtained by
slightly improving results by Dubuc \cite{Dubuc-Adjoint,Dubuc-KanExt} and by Tambara \cite{Tambara}. An advantage of our treatment in this
section is the fact that the construction of the adjoint $\overline{L}$ can
be given explicitly by means of a specific colimit in ${\Alg }({%
\Aa  })$.
In Section \ref{sec:relweakcoref}, we
consider the notion of weakly coreflective subcategory and we introduce a relative version of it in order to  reinterpret this colimit as a suitable initial object and obtain an explicit description for $\overline{L}\,\overline{B}$, for every algebra $\overline{B}$ in $\Bb  $. This turns out to be a useful tool to compare the initial objects in the involved categories. An instance of this fact is Proposition \ref{pro:initial}, where we prove that, if $\Cc$ is a replete posetal weakly coreflective full subcategory of a category $\Dd$, then $\Cc$ and $\Dd$ have the same initial objects, if any.
 Then we study pullbacks of a relative weakly coreflective subcategory along relative fibrations.
More precisely, in Proposition \ref{pro:pullbackweak}, we prove that, if $\Cc$ is a weakly $\Ee$-coreflective full subcategory of a category  $\Dd$, then the pullback $\Cc'$ of $\Cc$ along an $\Ee$-fibration $V:\Dd'\to\Dd$ is a weakly coreflective full subcategory of $\Dd'$. \begin{equation*}
\xymatrixrowsep{.6cm}\xymatrix{ \Cc'\pulb\ar@{^(->}[d] \ar[r]^-{U} &\Cc\ar@{^(->}[d] \\\Dd' \ar[r]^-{V} &\Dd}
\qquad \qquad
\xymatrixrowsep{.6cm}\xymatrix{ \epindalg(\overline{B})\pulb\ar@{^(->}[d] \ar[r]^-{U} &\epindobj(\overline{B})\ar@{^(->}[d] \\\indalg(\overline{B}) \ar[r]^-{V} &\indobj(\overline{B})}
\end{equation*}

These results will be applied in Section \ref{sec:crucial-initial} to the particular pullback represented in the diagram above, which  involves the categories $\indobj(\overline{B})$, $\epindobj(\overline{B})$, $\indalg(\overline{B}) $ and  $\epindalg(\overline{B}) $. We have mentioned that the functor $\overline{L}$ can be given explicitly in terms of a specific colimit. First, in Proposition \ref{pro:induced2}, we reinterpret this colimit as an initial object  in the category $\indalg(\overline{B}) $ of induced algebras of $\overline{B}$  leading us to Theorem \ref{teo:redinv}. Then, under suitable assumptions, we will see in Theorem \ref{teo:red} that $\indalg(\overline{B}) $ can be replaced by other three categories, precisely $\indobj(\overline{B})$, $\epindobj(\overline{B})$ and  $\epindalg(\overline{B}) $, that can be more easy to handle in practice, among them the category $\epindobj(\overline{B})$ of epi-induced objects. Then, in Proposition \ref{pro:constinepi}, we provide a construction of an initial object in $\epindobj(\overline{B})$. Putting together these results we obtain Proposition \ref{pro:brut} giving an explicit description for $\overline{L}\,\overline{B}$. By taking $\Cc^\op$ instead of $\Aa$, we get Proposition \ref{pro:brutop} that will be applied to the main example we are concerned in Section \ref{applicationsection}, namely the category $\Vec_G$ of $G$-graded vector spaces. This category is a particular instance of a pre-rigid category as it is monoidal closed. This led us to look for an analogue of Sweedler's finite dual in the general context of pre-rigid
braided monoidal categories. In \cite{AGM}, a monoidal category $\Cc $
is called pre-rigid if for every object $X$ there exists an object $X^{\ast
} $ and a morphism $\mathrm{ev}_{X}:X^{\ast }\otimes X\rightarrow \unit $
such that the map
\begin{equation*}
\hom _{\Cc }\left( T,X^{\ast }\right) \rightarrow \hom %
_{\Cc }\left( T\otimes X,\unit \right) :u\mapsto \mathrm{ev}%
_{X}\circ \left( u\otimes X\right)
\end{equation*}%
is bijective for every object $T$ in $\Cc $. In this framework, consider the functor $R:=(-)^*:\Cc ^{\op }\to \Cc $. It turns out, see Proposition \ref{prop:extensionGV} and Lemma \ref{lem:Barop}, that the adjunction $(L,R)$ is liftable,  whenever the functor $\overline{R}$ has a left adjoint. In
Proposition \ref{pro:Men0} we present conditions guaranteeing that this happens. Moreover, in Corollary \ref{coro:Men2} we find a pre-rigid analogue of \cite[Proposition 8]{PS}.

In Subsection \ref{GVec} we deal with the case when $\Cc $ is taken
to be the braided monoidal category $\Vec_G^\alpha$ of vector spaces graded by an abelian group $G$, where the braiding depends on a  bicharacter $\alpha:G\times G\to k\setminus\{0\}$ on $G$. In case $\alpha$ is skew-symmetric, our theory gives rise to
auto-adjunctions on the categories of bialgebras ``colored'' by $G$. As a
consequence of arguments settled in the slightly more general setting
in Remark \ref{rem:regop}, the lifted functors in this example can be described
explicitly. The paper concludes with hinting at
why one could expect that explicit descriptions as in case of $\Vec_G$
could be carried out, more generally, for the category of comodules over a
coquasi-bialgebra.

\section{Preliminaries and first results}\label{preliminares}
We begin our exposition by recalling some notions we need in the paper, among them the one of liftable pair of adjoint functors and its behaviour on braided categories, from \cite{GV-OnTheDuality}. Then we will present sufficient conditions for the functor induced by a lax monoidal right adjoint at the level of algebras to possess a left adjoint, see Theorem \ref{teo:Tambara}. This is obtained by slightly improving results by Dubuc and by Tambara.

\subsection{Some notational conventions}

When $X$ is an object in a category $\Cc $, we will denote the
identity morphism on $X$ by $1_X$ or $X$ for short. For categories $\mathcal{%
C}$ and $\mathcal{D}$, a functor $F:\Cc \to \mathcal{D}$ will be the
name for a covariant functor; it will only be a contravariant one if it is
explicitly mentioned. By $\mathsf{id}_{\Cc }$ we denote the identity
functor on $\Cc $. For any functor $F:\Cc \to \mathcal{D}$, we
denote $\mathsf{Id}_{F}$ (or sometimes -in order to lighten notation in some
computations- just $F$, if the context does not allow for confusion) the
natural transformation defined by $\mathsf{Id}_{FX}=1_{FX}$. \newline
Let $\Cc $ be a category. Denote by $\Cc ^{\op }$ the
opposite category of $\Cc $. Using the notation of \cite[page 12]%
{Pareigis-CatFunct}, an object $X$ and a morphism $f:X\rightarrow Y$ in $%
\Cc $ will be denoted by $X^{\op }$ and $f^{\op }:Y^{%
\op }\rightarrow X^{\op }$ when regarded as object and
morphism in $\Cc ^{\op }$. Given a functor $F:\Cc\to \Dd$, one defines its opposite functor $F^\op :\Cc^\op \to \Dd^\op $ by setting $F^\op X^\op =(FX)^\op $ and $F^\op f^\op =(Ff)^\op $. If $\alpha:F\to G$ is a natural transformation, its opposite $\alpha^\op$ is given by $(\alpha^\op)_{X^\op}:=(\alpha_X)^\op$ for every object $X$.\newline
\newline
Throughout the paper, we will work in the setting of monoidal categories.
With respect to the material presented below, it is useful to recall the
following notation. Let $(\mathcal{M},\otimes ,\unit ,a,l,r)$ be a
monoidal category. Following \cite[0.1.4, 1.4]{SaavedraRivano}, we have that
$\mathcal{M}^{\op }$ is also monoidal, the monoidal structure being
given by%
\begin{gather*}
X^\op\otimes Y^\op :=\left( X\otimes Y\right)
^\op,\qquad\text{ the unit object is }\unit ^\op \\
a_{X^\op,Y^\op,Z^\op} :=\left(
a_{X,Y,Z}^{-1}\right) ^{\op }, \qquad
l_{X^\op} :=\left( l_{X}^{-1}\right) ^{\op },\qquad
r_{X^\op}:=\left( r_{X}^{-1}\right) ^{\op }.
\end{gather*}%
If $\mathcal{M}$ is moreover braided (with braiding $c$), then so is $%
\mathcal{M}^{\op }$, the braiding being given by
\begin{equation*}
c_{X^\op ,Y^\op }:=\left( c_{X,Y}^{-1}\right) ^{%
\op }.
\end{equation*}
Unless explicitly stated, we will assume monoidal categories to be strict from now on. By Mac Lane's Coherence Theorem, this does not impose
restrictions on the obtained results. We will moreover consider braided
monoidal categories. A basic reference for these notions is \cite{MacLane},
for instance.\medskip

Recall (see e.g. \cite[Definition 3.1]{Aguiar-Mahajan})
that a functor $F:\Aa  \to\Bb  $ between monoidal categories $(%
\Aa  ,\otimes,{\unit }_{\Aa  })$ and $(\Bb  %
,\otimes^{\prime },{\unit }_{\Bb  })$ is said to be a lax monoidal
functor if it comes equipped with a family of natural morphisms $%
\phi_{2}(X,Y):F(X)\otimes^{\prime }F(Y)\to F(X\otimes Y)$, for $X,Y\in \mathcal{A%
}$, and a $\Bb  $-morphism $\phi_0:{\unit }_{\Bb  }\to F({%
\unit }_{\Aa  })$, satisfying the known suitable compatibility
conditions with respect to the associativity and unit constraints of $%
\Aa  $ and $\Bb  $.
Dually, colax monoidal functors are defined.\newline
Also recall that given a lax monoidal functor $(F,\phi_2,\phi_0)$, then $(F^{\op },\phi_2^{\op },\phi_0^{\op })$ is a colax monoidal functor, where we set $\phi_2^{\op }(X^{\op },Y^{\op }):=\phi_2(X,Y)^{\op }$, see e.g.  \cite[Proposition 3.7]{Aguiar-Mahajan}.

\subsection{Liftability of adjoint pairs}\label{sub:liftability}

Let $\left( L:\Bb  \rightarrow \Aa  ,R:\Aa  \rightarrow
\Bb  \right) $ be an adjunction with unit $\eta$ and counit $\epsilon$. It is known, see e.g. \cite[Proposition 3.84]%
{Aguiar-Mahajan}, that if $(L,\psi _{2},\psi _{0})$
is a colax monoidal functor, then $(R,\phi _{2},\phi _{0})$ is a lax
monoidal functor where, for every $X,Y\in \Aa  $,
\begin{gather*}
\phi _{2}\left( X,Y\right) :=\left(\xymatrixcolsep{35pt}\xymatrix{RX\otimes RY\ar[r]^-{\eta_{ \left(
RX\otimes RY\right)}}&RL\left( RX\otimes RY\right)\ar[r]^{R\psi _{2}\left( RX,RY\right)} &R\left( LRX\otimes
LRY\right)\ar[r]^-{R\left( \epsilon_{X}\otimes \epsilon_{Y}\right)} & R\left( X\otimes Y\right)}\right), \label{form:PhiFromPsi2}\\
\phi _{0} :=\left( \xymatrix{ \unit_{\Bb}\ar[r]^-{\eta_{ \unit_{\mathcal{%
B}}}}& RL\left( \unit_{\Bb}\right) \ar[r]^-{R\psi _{0}}&R\left( \unit_{\Aa}\right)} \right).   \label{form:PhiFromPsi0}
\end{gather*}%
Conversely, if $(R,\phi _{2},\phi _{0})$ is a lax monoidal functor, then $%
(L,\psi _{2},\psi _{0})$ is a colax monoidal functor where, for every $%
X,Y\in \Bb  $
\begin{gather}
\psi _{2}\left( X,Y\right) :=\left(\xymatrixcolsep{35pt}\xymatrix{ L\left( X\otimes Y\right) \ar[r]^-{%
L\left( \eta_{ X}\otimes \eta_{Y}\right) }&L\left( RLX\otimes
RLY\right) \ar[r]^-{L\phi _{2}\left( LX,LY\right) }%
&LR\left( LX\otimes LY\right) \ar[r]^-{\epsilon_{\left( LX\otimes LY\right) }}& LX\otimes LY}\right) ,   \label{form:PsiFromPhi2}\\
\psi _{0} :=\left( \xymatrix{L\left( \unit_{\Bb}\right) \ar[r]^-{L\phi _{0}%
}& LR\left( \unit_{\Aa}\right)\ar[r]^-{%
\epsilon_{\unit_{\Aa}}}&\unit_{\Aa%
}}\right) .  \label{form:PsiFromPhi0}
\end{gather}

Let $(R,\phi_2 ,\phi_0 ):\Aa  \rightarrow \Bb  $ be a lax
monoidal functor. It is well-known that $R$ induces a functor $\overline{R}:=%
{\Alg }(R):{\Alg }({\Aa  })\rightarrow {\Alg }({%
\Bb  })$ such that the diagram on the right-hand side in \eqref{diag:bar} commutes (cf. \cite[Proposition 6.1, page 52]%
{Benabou-IntrodBicat}; see also \cite[Proposition 3.29]{Aguiar-Mahajan}).
Explicitly,%
\begin{equation*}
\overline{R}\left( A,m,u\right) =\left( RA,\xymatrix{RA\otimes RA\ar[r]^-{\phi
_{2}\left( A,A\right) }&R\left( A\otimes A\right)\ar[r]^-{%
Rm}&RA},\xymatrix{\unit_{{\Bb}}\ar[r]^-{\phi _{0}}&R\unit_{{\Aa}}\ar[r]^-{Ru}&
RA}\right).
\end{equation*}%
Dually, a colax monoidal functor $(L,\psi _{2},\psi _{0}):\Bb  %
\rightarrow \Aa  $ colifts to a functor $\underline{L}:={\Coalg %
}(L):{\Coalg }({\Bb  })\rightarrow {\Coalg }({\Aa  %
})$ such that the diagram on the left-hand side in (\ref{diag:bar})
commutes. Explicitly,%
\begin{equation*}
\underline{L}\left( C,\Delta ,\varepsilon \right) =\left( LC,\xymatrix{LC\ar[r]^-{%
L\Delta }&L\left( C\otimes C\right) \ar[r]^-{\psi
_{2}\left( C,C\right) }&LC\otimes LC},\xymatrix{LC\ar[r]^-{%
L\varepsilon }&L\unit_{{\Bb}}\ar[r]^-{\psi _{0}%
}&\unit_{{\Aa}}}\right).
\end{equation*}%
The vertical arrows in the two diagrams below are the obvious forgetful
functors.

\begin{equation}
\begin{array}{ccc}
\xymatrix{{\Coalg}({\Bb  })\ar[d]_ {\mho '=\mho _{\Bb}
}\ar[rr]^-{\underline{L}=\Coalg(L)} && {{\Coalg}({\Aa  })}\ar[d]^{\mho =\mho
_{{\Aa  }}} \\ {\Bb  } \ar[rr]^-{L} && {\Aa} } &  &
\end{array}%
\qquad
\begin{array}{ccc}
\xymatrix{{\Alg}({\Aa  })\ar[d]_ {\Omega =\Omega
_{{\Aa  }}}\ar[rr]^-{\overline{R}=\Alg(R)} &&
{{\Alg}({\Bb  })}\ar[d]^{\Omega ' =\Omega _{{\Bb  }}} \\
{\Aa  } \ar[rr]^-{R} && {\Bb} } &  &
\end{array}
\label{diag:bar}
\end{equation}

\begin{definition}[{\protect\cite[Definition 2.3]{GV-OnTheDuality}}]
\label{def:liftable}Suppose $\Aa$ and $\Bb$  are monoidal
categories and $R:\Aa\rightarrow \Bb$ is a lax
monoidal functor with a left adjoint $L$. The pair $(L,R)$ is called \textit{liftable} if the induced functor $\overline{R}=\Alg(R):\Alg(\Aa)\to \Alg(\Bb)$ has a left adjoint\footnote{in general the left adjoint of $\overline{R}$ is not assumed to be of form $\Alg(L)$. In fact we can't even consider $\Alg(L)$ as $L$ needs not to be lax monoidal. A similar observation holds for the right adjoint of \underline{L}.}, say $\overline{L}$, and the induced functor $\underline{L}=\Coalg(L):\Coalg(\Bb)\to \Coalg(\Aa)$ has a right adjoint, say $\underline{R}$.
\end{definition}

Notice that being liftable really \textit{is} a condition: there exist
examples of lax monoidal functors $R$ between monoidal categories that have
a left adjoint $L$, but for which $\overline{R}$ does not have a left
adjoint. For instance, let $k$ be a field and set $S:=\frac{k \left[ X\right]%
}{\left( X^{2}\right) }$. Consider the functor
\begin{equation*}
R^{f}:\Vec ^{\text{f}}\rightarrow \Vec ^{\text{f}};V\mapsto S{%
\otimes}_{k} V.
\end{equation*}
In \cite[Example 4.2]{AGM}, it is shown that $\overline{R^{f}}$ has no left
adjoint.\medskip

\textbf{Liftability for braided monoidal categories.}
 Recall that when a monoidal category is \textit{braided}, its algebras and coalgebras inherit the monoidal structure, see e.g. \cite[1.2.2]{Aguiar-Mahajan}.
Let $\Aa  $ and $\Bb  $ now be \textit{%
braided} monoidal categories and let $R:\Aa  \rightarrow \Bb  $
be a braided lax monoidal functor having a left adjoint $L$. By e.g. \cite[%
Proposition 3.80]{Aguiar-Mahajan}, the functor $\overline{R}$ is lax
monoidal too. Explicitly, the lax monoidal functors $(R,\phi _{2},\phi _{0})$
and $(\overline{R},\overline{\phi }_{2},\overline{\phi }_{0})$ are connected
by the following equalities, for every $\overline{A}=\left(
A,m_{A},u_{A}\right) ,\overline{B}=\left( B,m_{B},u_{B}\right) \in {\mathsf{%
Alg}}({\Aa  })$
\begin{equation*}
\Omega _{{\Bb}}\circ \overline{R}=R\circ \Omega _{\mathcal{%
A}} ,\qquad \Omega _{{\Bb}} (\overline{\phi }%
_{2}\left( \overline{A},\overline{B}\right)) =\phi _{2}\left( A,B\right)
,\qquad  \Omega _{{\Bb}}(\overline{\phi }_{0})=\phi _{0}.
\label{form:phiOver}
\end{equation*}%
Note that $R$ is a braided lax monoidal functor if and only if $L$ is a braided colax monoidal functor, see e.g. \cite[Proposition 3.85]%
{Aguiar-Mahajan}. Moreover, if   $L$ is a braided colax monoidal functor one shows in a similar fashion as above that $%
\underline{L}$ is colax monoidal. The colax monoidal functors $(L,\psi
_{2},\psi _{0})$ and $(\underline{L},\underline{\psi }_{2},\underline{\psi }%
_{0})$ are connected by the following equalities for every $\underline{C}%
=\left( C,\Delta _{C},\varepsilon _{C}\right) ,\underline{D}=\left( D,\Delta
_{D},\varepsilon _{D}\right) \in {\Coalg }({\Bb  })$
\begin{equation*}
\mho _{\Aa}\circ \underline{L}=L\circ \mho _{{\Bb}%
} ,\qquad  \mho _{\Aa}( \underline{\psi }%
_{2}\left( \underline{C},\underline{D}\right) )=\psi _{2}\left( C,D\right)
,\qquad \mho _{\Aa}(\underline{\psi }_{0})=\psi _{0}.
\label{form:psiUnder}
\end{equation*}
As \cite[Example 4.2]{AGM} shows, a pair $(L,R)$, where $R:\Aa  %
\rightarrow \Bb  $ is a (braided) lax monoidal functor between
(braided) monoidal categories $\Aa  $ and $\Bb  $, having a left
adjoint $L$, needs not to be liftable, \textit{a priori}. But, in case $%
\Aa  $ and $\Bb  $ are braided monoidal categories and $R:%
\Aa  \rightarrow \Bb  $ is a braided lax monoidal functor having
a left adjoint $L$ such that the pair $(L,R)$ \textit{is} liftable, then, by
\cite[Lemma 2.4 and Theorem 2.7]{GV-OnTheDuality}, there is an adjunction $%
\left( \overline{\underline{L}},\overline{\underline{R}}\right) $ that fits
into the following commutative diagrams (and explains the choice of the
-perhaps somewhat fuzzy- term ``liftable'')
\begin{equation*}
\xymatrix{{\Bialg}({\Bb  })\ar[d]_
{\overline{\mho^{\prime}}}\ar[rr]^-{\overline{\underline{L}}=\Coalg(\overline{L})} &&
{{\Bialg}({\Aa  })}\ar[d]^{\overline{\mho '}} \\ {\Alg(\Bb  })
\ar[rr]^-{\overline{L}} && {\Alg(\Aa}) } \qquad \xymatrix{{\Bialg}({%
\Aa  })\ar[d]_ {\underline{\Omega}}\ar[rr]^-{\underline{\overline{R}}=\Alg(\underline{R})}
&& {{\Bialg}({\Bb  })}\ar[d]^{\underline{\Omega '}} \\
{\Coalg(\Aa  )} \ar[rr]^-{\underline{R}} && {\Coalg(\Bb}) }
\label{diag:bibar}
\end{equation*}
In this diagram, all vertical arrows are forgetful functors. 


\subsection{An approach to a result by Tambara, inspired by Dubuc}\label{tamDub}

In this subsection, we provide sufficient conditions (Theorem \ref{teo:Tambara} together with Proposition \ref{pro:Tambara}) for $%
\overline{R}$ to have a left adjoint. Under relatively mild assumptions,
this is obtained by considering a result by Dubuc \cite{Dubuc-Adjoint, Dubuc-KanExt} and by using it to provide a result in the spirit of
Tambara's \cite[Remark 1.5]{Tambara}, cf. \cite[Theorem 2.2.8]{AEM} for an unpublished proof of this result (Tambara does
not provide his own proof).
More precisely, let us compare the following two diagrams.

\begin{equation*}
\xymatrixrowsep{25pt}\xymatrix{\Alg({\Aa})\ar[rr]^-{\id_{\Alg(\Aa)}}
\ar@<.5ex>[d]^{\overline{R}} &&
{\Alg(\Aa  )}\ar@<.5ex>[d]^-{R{\Omega}} \\
{\Alg(\Bb  )}\ar[rr]_-{\Omega '}&&\Bb  \ar@<.5ex>[u]^{TL}}
\qquad \xymatrixrowsep{25pt}\xymatrixcolsep{55pt} \xymatrix{{\Aa}\ar[r]^{\id
_{\Aa}}\ar@<.5ex>[d]^{K} & \Aa \ar@<.5ex>[d]^{R}\\
{_{Q}\Bb  }\ar[r]_{U}&{\Bb}\ar@<.5ex>[u]^{L}}
\end{equation*}

\begin{itemize}
\item \textbf{The diagram on the left-hand side}. Here $\Aa  $ and $%
\Bb  $ are monoidal categories and it is assumed that the forgetful functor $%
\Omega :{\Alg }({\Aa  })\rightarrow {\Aa  }$ has a left
adjoint $T$. Also, $R:\Aa  \rightarrow \Bb  $ is supposed to be
a lax monoidal functor having a left adjoint $L$. \newline
If we are moreover given that $\Aa  $ has colimits and the tensor
product commutes with them, \cite[Remark 1.5]{Tambara} states that $%
\overline{R}$ has a left adjoint, too.

\item \textbf{The diagram on the right-hand side}. Here we are in the
setting of \cite[Theorem 1]{Dubuc-Adjoint} (see also \cite[Theorem A.1]%
{Dubuc-KanExt}) where, in case the category $\Aa  $ has reflexive coequalizers, the functor $K$ has a left adjoint {%
($U$ denotes the forgetful functor from the Eilenberg-Moore
category of algebras over the monad $Q$ on $\Bb  $ while $K$ is the
functor having a derivable adjoint triangle)}.
\end{itemize}

Thus, although the forgetful $\Omega ^{\prime }:{\Alg }%
({\Bb  })\rightarrow \Bb  $ (as above) is neither right adjoint\footnote{it does if ${\Bb  }$ has denumerable coproducts and they are
preserved by the tensor products, cf. \cite[Theorem 2, page 172]{MacLane}.} nor, equivalently, monadic (cf. \cite[Theorem A.6]{AM-MonadSym}), it
is still possible to produce a left adjoint for $\overline{R}$ like in the
diagram on the right-hand side, where instead $U$ is both right adjoint and
monadic. Moreover, on the right-hand side, just the existence of
reflexive coequalizers is required not of all colimits.

Inspired by Dubuc's work, we now present a result in the spirit of
Tambara's.
Note that there is no requirement on $\Omega $ here (no monadicity, neither
left adjoint). As a particular case we do, however, require that $\Aa  $ has all
coequalizers (not just reflexive coequalizers).

\begin{proposition}
\label{pro:TambaraGen}Consider the following diagram%
\begin{equation}
\xymatrix{{\Aa}\ar[r]^{\id _{\Aa}}\ar@<.5ex>[d]^{K} & \Aa
\ar@<.5ex>[d]^{R}\\
{\Alg(\Bb)}\ar@<.5ex>[u]^{\Lambda}\ar[r]_{\Omega}&{\Bb}\ar@<.5ex>[u]^{L}}
\label{diag:TambaraGen}
\end{equation}%
where $\Bb  $ is a monoidal category, $\Omega K=R$ and $\left(
L,R\right) $ is an adjunction with unit $\eta $ and counit $\epsilon $.
Given any algebra $\overline{B}:=\left( B,m_{B},u_{B}\right) $ in $\mathcal{B%
},$ write $KLB$ in the form $\left( RLB,m_{RLB},u_{RLB}\right)$ and assume
that the diagram
\begin{equation}
\xymatrixcolsep{2cm}\xymatrix{ L(B\ot B) \ar@<.5ex>[rr]^-{Lm_{B}}
\ar@<-.5ex>[rr]_-{{\epsilon}_{LB}\circ Lm_{RLB}\circ L\left( {\eta}_{
B}\otimes {\eta}_{ B}\right)}&& LB && L{\bf 1} \ar@<-.5ex>[ll]_-{Lu_{B}}
\ar@<.5ex>[ll]^-{{\epsilon}_{LB}\circ Lu_{RLB} }}  \label{diag:AlgDubuc}
\end{equation}%
has a colimit $\left( \Lambda \overline{B},\kappa _{\overline{B}%
}:LB\rightarrow \Lambda \overline{B}\right) $ (e.g. the category $\Aa
$ has coequalizers). This yields a functor $\Lambda$ which is a left adjoint
of the functor $K$.
Moreover the morphisms $\kappa_{\overline{B}}$ define a natural transformation $\kappa:L\Omega\to\Lambda$ such that $\kappa ={\epsilon }{\Lambda }\circ L\Omega {%
\overline{\eta }}$, where $\overline{\eta }$ denotes the unit of $(\Lambda,K)$.
\end{proposition}

\begin{proof}
Let $\overline{B}:=\left( B,m_{B},u_{B}\right) $ be an algebra in $\mathcal{B%
}$. From $\Omega K=R$, we can write $KLB$ in the form $\left(
RLB,m_{RLB},u_{RLB}\right) .$ By hypothesis, the diagram (\ref{diag:AlgDubuc}%
) has a colimit $\left( \Lambda \overline{B},\kappa _{\overline{B}%
}:LB\rightarrow \Lambda \overline{B}\right) $ (in case $\Aa  $ has
coequalizers, it is obtained by taking the coequalizer $\left( \Lambda
^{\prime }\overline{B},\kappa _{\overline{B}}^{\prime }:LB\rightarrow
\Lambda ^{\prime }\overline{B}\right) $ of the left-hand side pair and then
computing the coequalizer of the pair $\left( \kappa _{\overline{B}}^{\prime
}\circ Lu_{B},\kappa _{\overline{B}}^{\prime }\circ {\epsilon }_{LB}\circ
Lu_{RLB}\right) $). Let $f:\overline{B}\rightarrow \overline{E}$ be a
morphism in ${\Alg }({\Bb  })$. Then
\begin{eqnarray*}
L\Omega f\circ \epsilon _{LB}\circ Lm_{RLB}\circ L\left( {\eta }_{B}\otimes {%
\eta }_{B}\right) &=&{\epsilon }_{LE}\circ LRL\Omega f\circ Lm_{RLB}\circ
L\left( {\eta }_{B}\otimes {\eta }_{B}\right) \\
&=&{\epsilon }_{LE}\circ L\left[ \Omega KL\Omega f\circ m_{RLB}\circ \left( {%
\eta }_{B}\otimes {\eta }_{B}\right) \right] \\
&=&{\epsilon }_{LE}\circ L\left[ m_{RLE}\circ \left( RL\Omega f\otimes
RL\Omega f\right) \circ \left( {\eta }_{B}\otimes {\eta }_{B}\right) \right]
\\
&=&{\epsilon }_{LE}\circ Lm_{RLE}\circ L\left( {\eta }_{E}\otimes {\eta }%
_{E}\right) \circ L\left( \Omega f\otimes \Omega f\right) , \\
L\Omega f\circ Lm_{B} &=&L\left( \Omega f\circ m_{B}\right) =L\left(
m_{E}\circ \left( \Omega f\otimes \Omega f\right) \right) =Lm_{E}\circ
L\left( \Omega f\otimes \Omega f\right) , \\
L\Omega f\circ {\epsilon }_{LB}\circ Lu_{RLB} &=&{\epsilon }_{LE}\circ
LRL\Omega f\circ Lu_{RLB}={\epsilon }_{LE}\circ L\left( \Omega KL\Omega
f\circ u_{RLB}\right) ={\epsilon }_{LE}\circ Lu_{RLE}, \\
L\Omega f\circ Lu_{B} &=&L\left( \Omega f\circ u_{B}\right) =Lu_{E}
\end{eqnarray*}%
so that the following diagram serially commutes.%
\begin{equation*}
\xymatrixcolsep{2cm}\xymatrix{ L(B\ot B)\ar[d]_{L\left( \Omega f\otimes
\Omega f\right)} \ar@<.5ex>[rr]^-{Lm_{B}}
\ar@<-.5ex>[rr]_-{{\epsilon}_{LB}\circ Lm_{RLB}\circ L\left( {\eta}_{
B}\otimes {\eta}_{ B}\right)}&& LB \ar[d]_{L\left( \Omega f\right)} && L{\bf
1}\ar[d]_{\Id_{L\unit }} \ar@<-.5ex>[ll]_-{Lu_{B}}
\ar@<.5ex>[ll]^-{{\epsilon}_{LB}\circ Lu_{RLB} }\\ L(E\ot E)
\ar@<.5ex>[rr]^-{Lm_{E}} \ar@<-.5ex>[rr]_-{{\epsilon}_{LE}\circ
Lm_{RLE}\circ L\left( {\eta}_{ E}\otimes {\eta}_{ E}\right)}&& LE && L{\bf
1} \ar@<-.5ex>[ll]_-{Lu_{E}} \ar@<.5ex>[ll]^-{{\epsilon}_{LE}\circ Lu_{RLE}
} }.
\end{equation*}%
As a consequence, there is a unique morphism $\Lambda f:\Lambda \overline{B}%
\rightarrow \Lambda \overline{E}$ such that
\begin{equation}
\Lambda f\circ \kappa _{\overline{B}}=\kappa _{\overline{E}}\circ L\Omega f.
\label{form:Lambdaf}
\end{equation}%

Since $\kappa _{\overline{B}}$ is an epimorphism, one easily checks that $%
\Lambda \left( f\circ f^{\prime }\right) =\Lambda f\circ \Lambda f^{\prime }$
for all morphisms $f,f^{\prime }$ in ${\Alg }({\Bb  })$ that
can be composed. We thus get a functor $\Lambda :{\Alg }({\Bb  }%
)\rightarrow \Aa  $ and \eqref{form:Lambdaf} means that $\kappa_{-}$ is
natural in the lower argument. Let us check that $\left( \Lambda ,K\right) $ is
an adjunction. For every $A\in \Aa  $, consider the diagram (\ref%
{diag:AlgDubuc}) for $\overline{B}=KA$ (hence $B=RA$) i.e.

\begin{equation*}
\xymatrixcolsep{2cm}\xymatrix{ L(RA\ot RA) \ar@<.5ex>[rr]^-{Lm_{RA}}
\ar@<-.5ex>[rr]_-{{\epsilon}_{LRA}\circ Lm_{RLRA}\circ L\left( {\eta}_{
RA}\otimes {\eta}_{ RA}\right)}&& LRA && L{\bf 1} \ar@<-.5ex>[ll]_-{Lu_{RA}}
\ar@<.5ex>[ll]^-{{\epsilon}_{LRA}\circ Lu_{RLRA} }}.
\end{equation*}%
Then it is easily verified that ${\epsilon }_{A}\circ \left( {\epsilon }%
_{LRA}\circ Lm_{RLRA}\circ L\left( {\eta }_{RA}\otimes {\eta }_{RA}\right)
\right) ={\epsilon }_{A}\circ Lm_{RA}$ and that
${\epsilon }_{A}\circ \left( {\epsilon }_{LRA}\circ Lu_{RLB}\right)
={\epsilon }_{A}\circ Lu_{RA}$, so there exists a unique morphism ${%
\overline{\epsilon }}_{A}:\Lambda KA\rightarrow A$ such that
\begin{equation*}
{\overline{\epsilon }}_{A}\circ \kappa _{KA}={\epsilon }_{A}.
\end{equation*}%
If $h:A\rightarrow A^{\prime }$ is a morphism in $\Aa  $, we have
\begin{equation*}
{\overline{\epsilon }}_{A^{\prime }}\circ \Lambda Kh\circ \kappa _{KA}%
\overset{(\ref{form:Lambdaf})}{=}{\overline{\epsilon }}_{A^{\prime }}\circ
\kappa _{KA^{\prime }}\circ L\Omega Kh={\epsilon }_{A^{\prime }}\circ
LRh=h\circ {\epsilon }_{A}=h\circ \overline{{\epsilon }}_{A}\circ \kappa
_{KA}
\end{equation*}%
and hence ${\overline{\epsilon }}_{A^{\prime }}\circ \Lambda Kh=h\circ
\overline{{\epsilon }}_{A}$ which means that $\overline{\epsilon }_{-}$ is
natural in the lower argument. We have
\begin{eqnarray*}
R\kappa _{\overline{B}}\circ {\eta }_{B}\circ m_{B} &=&R\left( \kappa _{%
\overline{B}}\circ Lm_{B}\right) \circ {\eta }_{\left( B\otimes B\right) }=R%
\left[ \kappa _{\overline{B}}\circ {\epsilon }_{LB}\circ Lm_{RLB}\circ
L\left( {\eta }_{B}\otimes {\eta }_{B}\right) \right] \circ {\eta }_{\left(
B\otimes B\right) } \\
&=&R\left[ {\epsilon }_{\Lambda }\overline{B}\circ LR\kappa _{\overline{B}%
}\circ Lm_{RLB}\circ L\left( {\eta }_{B}\otimes {\eta }_{B}\right) \right]
\circ {\eta }_{\left( B\otimes B\right) } \\
&=&R{\epsilon }_{\Lambda \overline{B}}\circ {\eta }_{R\Lambda \overline{B}%
}\circ R\kappa _{\overline{B}}\circ m_{RLB}\circ \left( {\eta }_{B}\otimes {%
\eta }_{B}\right) \\
&=&R\kappa _{\overline{B}}\circ m_{RLB}\circ \left( {\eta }_{B}\otimes {\eta
}_{B}\right) \\
&=&m_{R\Lambda \overline{B}}\circ \left( R\kappa _{\overline{B}}\otimes R{%
\pi }_{\overline{B}}\right) \circ \left( {\eta }_{B}\otimes {\eta }%
_{B}\right) , \\
R\kappa _{\overline{B}}\circ {\eta }_{B}\circ u_{B} &=&R\left( \kappa _{%
\overline{B}}\circ RLu_{B}\right) \circ {\eta }_{\unit }=R\left( \kappa
_{\overline{B}}\circ {\epsilon }_{LB}\circ Lu_{RLB}\right) \circ {\eta }_{%
\unit } \\
&=&R\kappa _{\overline{B}}\circ R{\epsilon }_{LB}\circ {\eta }_{RLB}\circ
u_{RLB}=\Omega K\kappa _{\overline{B}}\circ u_{RLB}=u_{R\Lambda \overline{B}}
\end{eqnarray*}%
so that $R\kappa _{\overline{B}}\circ {\eta }_{B}$ induces an algebra map ${%
\overline{\eta }}_{\overline{B}}:\overline{B}\rightarrow K\Lambda \overline{B%
}$ such that%
\begin{equation*}
\Omega {\overline{\eta }}_{\overline{B}}=R\kappa _{\overline{B}}\circ {\eta }%
_{B}.
\end{equation*}%
Given a morphism $f:\overline{B}\rightarrow \overline{E}$ in ${\Alg }(%
{\Bb  })$, we get%
\begin{eqnarray*}
\Omega \left( {\overline{\eta }}_{\overline{E}}\circ f\right) &=&R\kappa _{%
\overline{E}}\circ {\eta }_{E}\circ \Omega f=R\left( \kappa _{\overline{E}%
}\circ L\Omega f\right) \circ {\eta }_{B} \\
&\overset{(\ref{form:Lambdaf})}{=}&R\left( \Lambda f\circ \kappa _{\overline{%
B}}\right) \circ {\eta }_{B}=R\Lambda f\circ \Omega {\overline{\eta }}_{%
\overline{B}}=\Omega \left( K\Lambda f\circ {\overline{\eta }}_{\overline{B}%
}\right)
\end{eqnarray*}%
and hence ${\overline{\eta }}_{\overline{E}}\circ f=K\Lambda f\circ {%
\overline{\eta }}_{\overline{B}}$ which means that $\overline{\eta }_{-}$ is
natural in the lower argument. We have that%
\begin{equation*}
{\overline{\epsilon }}_{\Lambda \overline{B}}\circ \Lambda {\overline{\eta }}%
_{\overline{B}}\circ \kappa _{\overline{B}}\overset{(\ref{form:Lambdaf})}{=}{%
\overline{\epsilon }}_{\Lambda \overline{B}}\circ \kappa _{K\Lambda
\overline{B}}\circ L\Omega {\overline{\eta }}_{\overline{B}}\newline
={\epsilon }_{\Lambda \overline{B}}\circ L\left( R\kappa _{\overline{B}%
}\circ {\eta }_{B}\right) =\kappa _{\overline{B}}\circ {\epsilon }_{LB}\circ
L{\eta }_{B}=\kappa _{\overline{B}},
\end{equation*}%
so that ${\overline{\epsilon }}_{\Lambda \overline{B}}\circ \Lambda {%
\overline{\eta }}_{\overline{B}}={\mathsf{Id}}_{\Lambda \overline{B}}$.
Moreover,
\begin{equation*}
\Omega \left( K{\overline{\epsilon }}_{A}\circ {\overline{\eta }}%
_{KA}\right) =R{\overline{\epsilon }}_{A}\circ \Omega {\overline{\eta }}%
_{KA}=R{\overline{\epsilon }}_{A}\circ R\kappa _{KA}\circ {\eta }_{RA}=R{%
\epsilon }_{A}\circ {\eta }_{RA}={\mathsf{Id}}_{RA}=\Omega {\mathsf{Id}}%
_{KA}.
\end{equation*}%
Since $\Omega $ is faithful, we get that $K{\overline{\epsilon }}_{A}\circ {%
\overline{\eta }}_{KA}={\mathsf{Id}}_{KA}.$ Thus $(\Lambda,K)$ is an adjunction. We compute%
\begin{equation*}
{\epsilon }_{\Lambda \overline{B}}\circ L\Omega {\overline{\eta }}_{%
\overline{B}}={\epsilon }_{\Lambda \overline{B}}\circ LR\kappa _{\overline{B}%
}\circ L{\eta }_{B}=\kappa _{\overline{B}}\circ {\epsilon }_{LB}\circ L{\eta
}_{B}=\kappa _{\overline{B}}.\qedhere
\end{equation*}%
\end{proof}

\begin{remark}
We already observed that, if $R:\Aa  \rightarrow \Bb  $ is a lax
monoidal functor having a left adjoint and if $\Aa  $ has colimits and
the tensor product commutes with them, then $\overline{R}$ is a right
adjoint too. It can be shown that 
the pair%
\begin{equation*}
\xymatrixcolsep{2cm}\xymatrix{ L(B\ot B) \ar@<.5ex>[rr]^-{Lm_{B}}
\ar@<-.5ex>[rr]_-{{\epsilon}_{{LB}}\circ Lm_{RLB}\circ L\left( {\eta}_{
B}\otimes {\eta}_{ B}\right)} &&LB}
\end{equation*}
is reflexive if we assume that ${\eta}_{ \unit }:\unit \rightarrow RL%
\unit $ is multiplicative.
\end{remark}

Let us fix the following setting we will frequently work in.

\begin{notations}
\label{not:setting}Let $%
\Aa  $ and $\Bb  $ be monoidal categories and let $R:{\Aa  %
}\rightarrow \Bb  $ be a lax monoidal functor with a left adjoint $L,$
unit $\eta :\id_{\Bb  }\rightarrow RL$ and counit $\epsilon :LR\rightarrow \id_{\Aa  }$. Assume that the forgetful functor $\Omega :{%
\Alg }({\Aa  })\rightarrow {\Aa  }$ has a left adjoint $%
T, $ with unit $\alpha :\mathrm{id}_{\Aa  }\rightarrow \Omega T$ and counit $\gamma :T\Omega
\rightarrow \id_{\Alg (\Aa  )} $.

Given an algebra $\overline{B}=\left( B,m_{B},u_{B}\right) $ in $\Bb
$, write
\begin{equation}
\left( R\Omega TLB,m_{R\Omega TLB},u_{R\Omega TLB}\right) =\overline{R}TLB,
\label{eq:newlabel}
\end{equation}%
and set $\mu :=\gamma _{TLB}\circ T\epsilon _{\Omega TLB}\circ TLm_{R\Omega
TLB}\circ TL\left( R\alpha _{LB}\circ \eta _{B}\otimes R\alpha _{LB}\circ
\eta _{B}\right)$. Consider the diagram
\begin{equation}
\xymatrixcolsep{2cm}\xymatrix{ TL(B\ot B) \ar@<.5ex>[rr]^-{TLm_{B}}
\ar@<-.5ex>[rr]_-{\mu}&& TLB && TL{\bf 1} \ar@<-.5ex>[ll]_-{TLu_{B}}
\ar@<.5ex>[ll]^-{\gamma _{TLB}\circ T\epsilon _{\Omega TLB}\circ
TLu_{R\Omega TLB}}} .  \label{eq:14bis}
\end{equation}
\end{notations}

The following result provides a sufficient condition for
$\overline{R}=\Alg(R)$ to have a left adjoint for a lax monoidal functor $R$  with a left adjoint.

\begin{theorem}
\label{teo:Tambara} In the Setting \ref{not:setting}, assume
that for every algebra $\overline{B}$ in $\Bb  $ the diagram %
\eqref{eq:14bis} has a colimit $\left( \overline{L}\,\overline{B},\kappa _{%
\overline{B}}:TLB\rightarrow \overline{L}\,\overline{B}\right) $ (e.g. the
category ${\Alg }({\Aa  })$ has coequalizers). This yields a
functor $\overline{L}$ which is a left adjoint of the functor $\overline{R}=\Alg(R):{%
\Alg }({\Aa  })\rightarrow {\Alg }({\Bb  })$ and the morphisms $\kappa_{\overline{B}}$ define a natural transformation $\kappa:TL\Omega\to \overline{L}$.
\end{theorem}

\begin{proof}
Let $\left( L,R,\eta ,\epsilon\right) $ and $\left( T,\Omega
,\alpha ,\gamma \right) $ be the
adjunctions as in Setting \ref{not:setting}. Their composition yields the adjunction
\begin{equation*}
\left( TL,R\Omega ,\mathrm{id}_{\Bb  }\overset{\eta }{\longrightarrow }%
RL\overset{R\alpha L}{\longrightarrow }R\Omega TL,TLR\Omega \overset{%
T\epsilon \Omega }{\longrightarrow }T\Omega \overset{\gamma }{%
\longrightarrow }\mathrm{id}_{\Alg (\Aa  )}\right) .
\end{equation*}%
Then diagram (\ref{diag:AlgDubuc}) becomes $\left( \ref{eq:14bis}\right)$ in
our setting as the role of diagram \eqref{diag:TambaraGen} is played by the
following diagram%
\begin{equation}
\xymatrix{\Alg({\Aa})\ar[rr]^-{\id} \ar@<.5ex>[d]^{\overline{R}} &&
{\Alg(\Aa  )}\ar@<.5ex>[d]^-{R{\Omega}} \\
{\Alg(\Bb  )}\ar[rr]_-{\Omega '}&&\Bb  \ar@<.5ex>[u]^{TL} }
\label{diag:Tambara-1}
\end{equation}%
where $\Omega ^{\prime }\overline{R}=R\Omega .$ The conclusion follows by
Proposition \ref{pro:TambaraGen}.
\end{proof}

The next result collects sufficient conditions for Theorem \ref{teo:Tambara} to be applied.

\begin{proposition}\label{pro:Tambara}
For a monoidal category $\Aa$, assume that ${\Alg}(\Aa)$ is complete, well-powered and it has a cogenerating family. Then the category ${\Alg}(\Aa)$ has coequalizers. Moreover, the forgetful functor $\Omega :{\Alg}(\Aa)\rightarrow \Aa$ has a left adjoint $T$.
\end{proposition}

\begin{proof}
By \cite[Proposition 3.3.8]{Borceux1}, the category ${\Alg}({\Aa  })$ is cocomplete. In particular, ${\Alg}({\Aa  })$ has coequalizers. Moreover, since $\Omega$ creates limits (cf. \cite[Proposition 2.5]{Pareigis-3}), it also preserves them so that by the special adjoint functor theorem (cf. \cite[Theorem 3.3.4]{Borceux1}) it has a left adjoint $T$.
\end{proof}

\begin{example}\label{exa:aureo}
Let $k$ be a commutative ring. Let $\Aa  =k\text{-}{\sf Mod}^\op$ be the opposite of the category $k\text{-}{\sf Mod}$.
Thus  ${\Alg}({\Aa  })=k\text{-}\Coalg^\op$ is the opposite of the category $k$-$\Coalg=\Coalg(k\text{-}{\sf Mod})$ of $k$-coalgebras and their morphisms.
By the proof of \cite[Theorem 4.1]{Ba}, the category ${\Alg}({\Aa  })$ is complete, well-powered and it has a cogenerating family.
Thus Proposition \ref{pro:Tambara} applies. As a consequence, by Theorem \ref{teo:Tambara}, for every lax monoidal functor $R:k\text{-}{\sf Mod}^\op\to \Bb$ with a left adjoint, the functor $\overline{R}=\Alg(R):k\text{-}\Coalg^\op\to \Alg(\Bb)$ has a left adjoint too.
\end{example}

In Theorem \ref{teo:Tambara}, the existence of a colimit for diagram \eqref{eq:14bis} plays a crucial role. In Section \ref{sec:crucial-initial} we will reinterpret this colimit as a suitable initial object obtaining an explicit description for $\overline{L}\,\overline{B}$, for every algebra $\overline{B}$ in $\Bb  $, under relevant assumptions. The aforementioned reinterpretation is based on the notions of relative weak coreflections and fibrations, which is our next topic of investigation.

\section{Relative weak coreflections and fibrations}\label{sec:relweakcoref}

In this section we consider the notion of weakly coreflective subcategory and a relative version of it as a tool to compare the initial objects in the involved categories, obtaining Corollary \ref{coro:constrinitial} and Proposition \ref{pro:initial}. Then we study pullbacks of a relative weakly coreflective subcategory along relative fibrations, see Proposition \ref{pro:pullbackweak}. These results will be used in Section \ref{sec:crucial-initial} in order to prove Proposition \ref{prop:initialepiobj}, Proposition \ref{prop:initialepi} and Proposition \ref{pro:constinepi}.

\subsection{Relative weak coreflections}
Consider a full subcategory $\Cc$ of a category $\Dd$. By a \emph{weak coreflection} of an object $D\in\Dd$ in $\Cc$ we mean a morphism $r_D:D^\star\to D$ in $\Dd$ with $D^\star\in\Cc$ and such that the function $\hom_\Dd(C,r_D):\hom_\Dd(C,D^\star)\to \hom_\Dd(C,D)$ is surjective for all $C\in\Cc$.
Given a class $\Ee$ of morphisms in $\Dd$, then $\Cc$ is said to be \emph{weakly $\Ee$-coreflective} if each object in $\Dd$ has a weak coreflection $r_D\in \Ee$. If $\Ee$ is the whole class of morphisms in $\Dd$ we will just say \emph{weakly coreflective}, see e.g. \cite[Definition 4.5]{AR}. Clearly, a weakly $\Ee$-coreflective subcategory is in particular weakly coreflective.

\begin{remark}
When $\Ee$ is the class of monomorphisms (resp. epimorphism) one could speak about weakly mono-coreflective (epi-coreflective), in analogy to the non-weak case, see e.g. \cite{HS}. Anyway, we will not deal with these cases.
\end{remark}

Consider a weakly coreflective subcategory $\Cc$ of a category $\Dd$ and let $J:\Cc\to\Dd$ be the canonical embedding.
Then the object function $(-)^\star:{\sf Obj}(\Dd)\to {\sf Obj}(\Cc),\,D\mapsto D^\star,$ yields a \emph{weak right adjoint} to the functor $J$, see \cite{Ma} (where it is called a right adjoint system). The item $1)$ in the following result, proved under the further assumption that the category $\Cc$ is posetal, is an analogue, for this particular weak right adjoint, of the well-known fact that a right adjoint preserves limits.

\begin{proposition}\label{prop:limweak}
Let $\Cc$ be a posetal weakly coreflexive full subcategory of a category  $\Dd$ and let $J:\Cc\to\Dd$ be the canonical embedding. Given a functor $F:\Aa\to \Cc$, the following assertions holds.
\begin{enumerate}
  \item[$1)$] If $(L,(l_A)_{A\in \Aa})$ is a limit of $JF$, then $(L^\star,(l_A\circ r_L)_{A\in \Aa})$ is a limit of $F$.
  \item[$2)$] If $(L,(l_A)_{A\in \Aa})$ is a colimit of $JF$, then $(L^\star,(l_A^\star)_{A\in \Aa})$ is a colimit of $F$, for a unique morphism $l_A^\star:FA\to L^\star$ in $\Cc$ such that $r_L\circ l_A^\star=l_A$, for every $A\in\Aa$.
\end{enumerate}
\end{proposition}

\begin{proof}
$1)$ Set $l_A^\star:=l_A\circ r_L:L^\star\to FA$ which is clearly a morphism in $\Cc$. Given a morphism $a:A\to A'$ in $\Aa$, we have that $Fa\circ l_A^\star=JFa\circ l_A\circ r_L=l_{A'}\circ r_L=l_{A'}^\star$ so that $(L^\star,(l_A^\star)_{A\in \Aa})$ is a cone on $F$. Given another cone $(C,(c_A)_{A\in \Aa})$ on $F$, it is in particular a cone on $JF$ so that, since $(L,(l_A)_{A\in \Aa})$ is a limit of $JF$, there is a unique morphism $c:C\to L$ in $\Dd$ such that $l_A\circ c=c_A$. Since $\hom_\Dd(C,r_L):\hom_\Dd(C,L^\star)\to \hom_\Dd(C,L)$ is surjective, there is $c':C\to L^\star$ in $\Dd$ such that $r_L\circ c'=c$. Thus $l_A^\star\circ c'=l_A\circ r_L\circ c'=l_A\circ c=c_A$. Finally, since $C$ and $L^\star$ are in $\Cc$, the morphism $c'$ is in fact a morphism in $\Cc$ and hence it is unique as $\Cc$ is posetal.

$2)$ Since $\hom_\Dd(FA,r_L):\hom_\Dd(FA,L^\star)\to \hom_\Dd(FA,L)$ is surjective, there is $l_A^\star:FA\to L^\star$ in $\Dd$ such that $r_L\circ l_A^\star=l_A$. Since $FA$ and $L^\star$ are in $\Cc$, the morphism $l_A^\star$ is in fact a morphism in $\Cc$ and hence it is unique as $\Cc$ is posetal. Clearly $(L^\star,(l_A^\star)_{A\in \Aa})$ is automatically a cocone on $F$ as $\Cc$ is posetal. Given another cocone $(C,(c_A)_{A\in \Aa})$ on $F$, it is in particular a cocone on $JF$ so that, since $(L,(l_A)_{A\in \Aa})$ is a colimit of $JF$, there is a unique morphism $c:L\to C$ in $\Dd$ such that $c\circ l_A=c_A$. Set $c^\star:=c\circ r_L:L^\star\to C$. This is a morphism in $\Cc$ whence it is unique as $\Cc$ is posetal. Moreover $c^\star\circ l_A^\star=c\circ r_L\circ l_A^\star=c\circ  l_A=c_A$.
\end{proof}

%

The following result will be useful in constructing explicitly an initial object in a posetal weakly coreflexive full subcategory.

\begin{corollary}\label{coro:constrinitial}
Let $\Cc$ be a posetal weakly coreflexive full subcategory of a category  $\Dd$. Assume there is a set $\Ss$ consisting of objects in $\Cc$ such that each object in $\Cc$ is isomorphic to an element of $\Ss$. If there exists the product $\prod_{S\in \Ss}S$ in $\Dd$, then $(\prod_{S\in \Ss} S)^\star$ is an initial object in $\Cc$.
\end{corollary}

\begin{proof}
Set $D:=\prod_{S\in \Ss}S\in\Dd$. By Proposition \ref{prop:limweak}, we have that $D^\star\in\Cc$ is the product of the elements of $\Ss$ in $\Cc$ so that we can consider the canonical projection $p_S:D^\star\to S$ in $\Cc$, for every $S\in\Ss$. Given $C\in \Cc$, there is $S\in\Ss$ and an isomorphism $f:S\to C$ in $\Cc$. Since $\Cc$ is posetal we get $\hom_\Cc(D^\star,C)=\{f\circ p_S\}$ and hence $D^\star$  is an initial object in $\Cc$.
\end{proof}

\begin{lemma}\label{lem:weakunique}
Let $\Cc$ be a posetal weakly coreflective full subcategory of a category $\Dd$. Then, for every $C\in\Cc$ and $D\in\Dd$, there can be a unique morphism $C\to D$ in $\Dd$.
\end{lemma}

\begin{proof}
Since $\Cc$ is weakly coreflective,  there is a morphism $r_D:D^\star\to D$ with $D^\star\in\Cc$ and such that the function $\hom_\Dd(C,r_D):\hom_\Dd(C,D^\star)\to \hom_\Dd(C,D)$ is surjective. Since $\Cc$ is a posetal full subcategory of a category $\Dd$, we have that $\hom_\Dd(C,D^\star)=\hom_\Cc(C,D^\star)$ has at most one element so that $\hom_\Dd(C,D)$  has at most one element.
\end{proof}

\begin{proposition}\label{pro:initial}
  Let $\Cc$ be a replete posetal weakly coreflective full subcategory of a category $\Dd$.
   Then $\Cc$ and $\Dd$ have the same initial objects, if any.
 \end{proposition}
\begin{proof}
Assume that $I$ is an initial object in $\Cc$ and let $D\in \Dd $. By Lemma \ref{lem:weakunique}, the set $\hom_\Dd(I,D)$ has at most one element. Since $\Cc$ is weakly coreflective,  there is a morphism $r_D:D^\star\to D$ with $D^\star\in\Cc$ and since $I$ is initial in $\Cc$ there is a morphism $i:I\to D^\star$ so that $r_D\circ i\in \hom_\Dd(I,D)$ and hence $\hom_\Dd(I,D)=\{r_D\circ i\}$ so that $I$ is initial in $\Dd$.

Conversely, assume $I$ is an initial object in $\Dd$. Since $I$ is a particular instance of colimit, by Proposition \ref{prop:limweak}, we have that $I^\star$ is an initial object in $\Cc$. By the foregoing, $I^\star$ is an initial object also in $\Dd$. By uniqueness we get $I\cong I^\star$ as objects in $\Dd$. Since $\Cc$ is replete and $I^\star\in\Cc$, we get that $I\in \Cc$. Thus $I$ is an initial object also in $\Cc$.
\end{proof}

Let $\Cc$ be a full subcategory of a category $\Dd$. We can consider the pullback of $\Cc$ along a functor $V:\Dd'\to\Dd$ i.e. the full subcategory $\Cc'$ of $\Dd'$ whose objects are the $D'\in\Dd'$ such that $VD'\in\Cc$. Clearly $V$ induces the functor $U:\Cc'\to\Cc,C'\mapsto VC',f'\mapsto Vf'$ which makes commute the diagram
\begin{equation*}
\xymatrixrowsep{.6cm}\xymatrix{ \Cc'\pulb\ar@{^(->}[d] \ar[r]^-{U} &\Cc\ar@{^(->}[d] \\\Dd' \ar[r]^-{V} &\Dd}
\end{equation*}
The instance of this situation we are interested in is the diagram in Remark \ref{rem:functors}.

\begin{lemma}\label{lem:pullback}Let $\Cc$ be a replete posetal full subcategory of a category  $\Dd$. Consider the pullback $\Cc'$ of $\Cc$ along a faithful functor $V:\Dd'\to\Dd$. Then $\Cc'$ is a replete posetal full subcategory of $\Dd'$.
\end{lemma}

\begin{proof}
Let $C'\in\Cc'$ and $D'\in\Dd'$. Given an isomorphism $h':C'\to D'$ in $\Dd'$, then we get an isomorphism $Vh':VC'\to VD'$ in $\Dd$. Since $VC'\in\Cc$, and $\Cc$ is replete, we get that $VD'\in\Cc$ and hence $D'\in \Cc'$. Thus $\Dd'$ is a replete full subcategory of $\Cc'$.

Given morphisms $f,g:C'\to C''$ in $\Cc'$, we get that $Vf,Vg:VC'\to VC''$ are morphisms in $\Cc$. Since $\Cc$ is posetal we get $Vf=Vg$. Since $V$ is faithful, we get $f=g$ and hence $\Cc'$ is posetal.
\end{proof}

\subsection{Relative fibrations}
Let $F:\mathcal{A}\rightarrow \mathcal{B}$ be a functor. Recall that a morphism $f\in \mathcal{A}$ is \emph{cartesian} (with
respect to $F$) over a morphism $f^{\prime }\in \mathcal{B}$ whenever $%
Ff=f^{\prime }$ and, when given $g\in \mathcal{A}$ and $h\in \mathcal{B}$ such
that $Ff\circ h=Fg,$ there exists a unique morphism $k\in \mathcal{A}$
such that $Fk=h$ and $f\circ k=g$.
\begin{equation*}
\xymatrixcolsep{1.5cm}\xymatrixrowsep{0.5cm}\xymatrix{&FZ\ar[dl]_h\ar[d]^{Fg}\\ FX\ar[r]|{Ff}&FY }\qquad\qquad
\xymatrixcolsep{1.5cm}\xymatrixrowsep{0.5cm}\xymatrix{&Z\ar@{.>}[dl]_k\ar[d]^{g}\\ X\ar[r]|{f}&Y }
\end{equation*}
Let $\Ee$ be a class of morphisms in $\mathcal{B}$. We say that $F$ is an $\Ee$-\emph{fibration} if every morphism $%
f^{\prime }:B\rightarrow FA$ in $\Ee$ there is $f:A^{\prime
}\rightarrow A\ $which is cartesian over $f^{\prime }$, see \cite[Definition 4.1]{AM-MonBrAdj}.

\begin{proposition}\label{pro:pullbackweak}Let $\Cc$ be a weakly $\Ee$-coreflexive full subcategory of a category  $\Dd$. Consider the pullback $\Cc'$ of $\Cc$ along an $\Ee$-fibration $V:\Dd'\to\Dd$. Then $\Cc'$ is a weakly coreflexive full subcategory of $\Dd'$.
\end{proposition}

\begin{proof}
Given $X\in\Dd'$, since $VX\in\Dd$, we can consider $r_{VX}:(VX)^\star\to VX$ in $\Ee$. Since $V$ is an $\Ee$-fibration, there is a morphism $r'_{X}:X^\star\to X$ in $\Dd'$ which is cartesian over $r_{VX}$. In particular $V(r'_{X})=r_{VX}$ so that $V(X^\star)=(VX)^\star$ and hence $X^\star\in\Cc'$. We have to check that $\hom_{\Dd'}(C',r'_X)$ is surjective for all $C'\in\Cc'$. 
\begin{equation*}
\xymatrixcolsep{2.5cm}\xymatrix{ \hom_{\Dd'}(C',X^\star)\ar[d] \ar[r]^-{\hom_{\Dd'}(C',r'_X)} &\hom_{\Dd'}(C',X)\ar[d] \\
\hom_{\Dd}(VC',V(X^\star)) \ar@{->>}[r]^-{\hom_{\Dd}(VC',V(r'_{X}))} &\hom_{\Dd}(VC',VX)}
\end{equation*}
Let $f\in\hom_{\Dd'}(C',X)$. Then $Vf\in \hom_{\Dd}(VC',VX)$. Since $\hom_{\Dd}(VC',r_{VX})$ is surjective, there is $g\in\hom_{\Dd}(VC',(VX)^\star)$ such that $r_{VX}\circ g=Vf$. Since $r'_{X}$ is cartesian over $r_{VX}$, there is a unique morphism $g'\in\hom_{\Dd'}(C',X^\star)$ such that $Vg'=g$ and $r'_{X}\circ g'=f$. Thus $\hom_{\Dd'}(C',r'_X)$ is surjective.
\end{proof}

\section{The crucial colimit as an initial object}\label{sec:crucial-initial}

In Theorem \ref{teo:Tambara} it is shown that, in the Setting \ref{not:setting}, the existence of a colimit for the diagram \eqref{eq:14bis}, for every algebra $\overline{B}$  in $\Bb  $, yields a functor $\overline{L}$ which is a left adjoint of the functor $\overline{R}=\Alg(R):{%
\Alg }({\Aa  })\rightarrow {\Alg }({\Bb  })$. The first aim of this section is to reinterpret this colimit as an initial object  in the category $\indalg(\overline{B}) $ of induced algebras of $\overline{B}$. This will lead us to rewrite Theorem \ref{teo:Tambara} as Theorem \ref{teo:redinv}. Then, under suitable assumptions, in several steps we will see in Theorem \ref{teo:red} that $\indalg(\overline{B}) $ can be replaced by three other categories that can be more easy to handle in practice, among them the category $\epindobj(\overline{B})$ of epi-induced objects. Then we will provide a construction of an initial object in $\epindobj(\overline{B})$. Putting together these results we will provide Proposition \ref{pro:brut} giving an explicit description for $\overline{L}\,\overline{B}$. By taking $\Cc^\op$ instead of $\Aa$ we will get Proposition \ref{pro:brutop} that will be used together with Remark \ref{rem:regop} in the Section \ref{applicationsection} for our main example.

\subsection{Induced objects and algebras}
Our aim here is to characterize a colimit for  \eqref{eq:14bis} as an initial object in the category $\indalg(\overline{B}) $ of induced algebras of $\overline{B}$.

%

\begin{definition}
Let $\left( L:\Bb  \rightarrow \Aa  ,\psi _{2},\psi _{0}\right) $
be a colax monoidal functor and let $\overline{B}:=\left(
B,m_{B},u_{B}\right) $ be an algebra in $\Bb  $.
We say that $\left( \overline{A},q\right) $ is an \textit{induced object of }%
$\overline{B}$ (by $L$) whenever $\overline{A}=\left( A,m_{A},u_{A}\right) $
consists of an object $A$ and morphisms $m_{A}:A\otimes A\rightarrow A,$ $%
u_{A}:\unit \rightarrow A$ and $q:LB\rightarrow A$ in $\Aa  $ such
that
\begin{equation*}\xymatrixcolsep{2cm}
\xymatrix{L(B\otimes B)\ar[d]^{\psi _{2}\left(
B,B\right)}\ar[rr]^{Lm_B}&&LB\ar[d]_q\\
LB\otimes LB\ar[r]^{q\otimes q}&A\otimes A\ar[r]^{m_A}&A
}\qquad
\xymatrix{L\unit \ar[d]^{\psi _{0}}\ar[r]^{Lu_B}&LB\ar[d]_q\\
\unit \ar[r]^{u_A}&A
}.
\end{equation*}
A \emph{morphism} $h:(\overline{A},q)\to (\overline{A'},q') $ of induced objects of $\overline{B}$ is a morphism $h:A\to A'$ such that $h\circ q=q'$. In this way we have defined the category $\indobj(\overline{B}) $ of induced objects of $\overline{B}$ and their morphisms.
Given $(\overline{A},q)$ in $\indobj(\overline{B})$,
 if the triple $\overline{A}=\left( A,m_{A},u_{A}\right) $ is an algebra in $%
\Aa  $ then $\left( \overline{A},q\right) $ is called an \textit{%
induced algebra of }$\overline{B}$ (by $L$). Note that $\left( \overline{A}%
,q\right) $ is an object in the comma category $\left( LB\downarrow \Omega
\right) ,$ see \cite[page 47]{MacLane}, where $\Omega :\Alg \left(
\Aa  \right) \rightarrow \Aa  $ is the forgetful functor. Thus we can define a \emph{morphism} $\overline{h}:(\overline{A},q)\to (\overline{A'},q') $ of induced algebras of $\overline{B}$ to be an algebra morphism $\overline{h}:\overline{A}\to \overline{A'}$ such that $\Omega\overline{h}\circ q=q'$. In this way we have defined the category $\indalg(\overline{B}) $ of induced algebras of $\overline{B}$ an their morphisms. %
\end{definition}

\begin{remark}
The two above notions of induced object and algebra
already appeared in \cite[Definition 11]{PS} with a slightly different
terminology.
\end{remark}

%

We now turn to the Setting \ref{not:setting}.
It is well-known that the colimit of a diagram $D$ is the initial object in the category formed by cocones on $D$.
\begin{invisible} See, e.g. Definition 5.16 in [Awodey, "Category Theory"]\end{invisible}
In particular, we get that a colimit for diagram \eqref{eq:14bis} is an initial object in the category $\cocone(\overline{B})$ whose objects are pairs $(\overline{A},\xi)$, where  $\xi :TLB\rightarrow \overline{A}$ is an algebra morphism that coequalizes the pairs in \eqref{eq:14bis} and whose morphisms $(\overline{A},\xi)\to (\overline{A'},\xi')$ are algebra morphisms $\overline{f}:\overline{A}\to \overline{A'}$ such that $\overline{f}\circ \xi=\xi'$.

 The next aim is to use the adjunction $(T,\Omega,\alpha:\id_\Aa\to\Omega T,\gamma :T\Omega
\rightarrow \id_{\Alg (\Aa  )})$ to  show that the category $\cocone(\overline{B})$ is isomorphic to the category $\indalg(\overline{B})$ so that the respective initial objects are in bijective correspondence.


\begin{proposition}
\label{pro:induced1} In the Setting \ref{not:setting}, let $%
\overline{B}:=\left( B,m_{B},u_{B}\right) $ be an algebra in $\Bb  .$
Then an algebra morphism $\xi :TLB\rightarrow \overline{A}$ coequalizes the
pairs in \eqref{eq:14bis} if and only if $\left( \overline{A}%
,\Omega \xi \circ \alpha _{LB}:LB\rightarrow A\right) $ is an induced
algebra of $\overline{B}$. As a consequence we get the category isomorphism
$$F:\cocone(\overline{B})\to \indalg(\overline{B}),\quad (\overline{A},\xi)\mapsto ( \overline{A}%
,\Omega \xi \circ \alpha _{LB}),\quad \overline{f}\mapsto \overline{f}$$
whose inverse $F^{-1}$ is given by $F^{-1}(\overline{A},q):=(\overline{A},\gamma_{\overline{A}}\circ Tq)$.
\end{proposition}

\begin{proof}
Recall that the morphisms $m_{R\Omega TLB}$ and $u_{R\Omega TLB}$ are
determined by the equality \eqref{eq:newlabel} so that
\begin{equation*}
m_{R\Omega TLB} = Rm_{\Omega TLB}\circ \phi _{2}\left( \Omega TLB,\Omega
TLB\right) ~\text{and}~ u_{R\Omega TLB}= Ru_{\Omega TLB}\circ \phi _{0}
\end{equation*}
where $TLB=\left( \Omega TLB,m_{\Omega TLB},u_{\Omega TLB}\right) .$ By
using this fact, we want to rewrite some of the morphisms in $\left( \ref%
{eq:14bis}\right) $. We have%
\begin{eqnarray*}
&&\gamma _{TLB}\circ T\epsilon _{\Omega TLB}\circ TLm_{R\Omega TLB}\circ
TL\left( R\alpha _{LB}\circ \eta _{B}\otimes R\alpha _{LB}\circ \eta
_{B}\right) \\
&=&\gamma _{TLB}\circ Tm_{\Omega TLB}\circ T\epsilon _{\Omega TLB\otimes
\Omega TLB}\circ TL\phi _{2}\left( \Omega TLB,\Omega TLB\right) \circ
TL\left( R\alpha _{LB}\otimes R\alpha _{LB}\right) \circ TL\left( \eta
_{B}\otimes \eta _{B}\right) \\
&=&\gamma _{TLB}\circ Tm_{\Omega TLB}\circ T\left( \alpha _{LB}\otimes
\alpha _{LB}\right) \circ T\epsilon _{LB\otimes LB}\circ TL\phi _{2}\left(
LB,LB\right) \circ TL\left( \eta _{B}\otimes \eta _{B}\right)
\end{eqnarray*}%
and%
\begin{eqnarray*}
\gamma _{TLB}\circ T\epsilon _{\Omega TLB}\circ TLu_{R\Omega TLB} &=&\gamma
_{TLB}\circ T\epsilon _{\Omega TLB}\circ TLRu_{\Omega TLB}\circ TL\phi _{0}
\\
&=&\gamma _{TLB}\circ Tu_{\Omega TLB}\circ T\epsilon _{\unit _{\mathcal{A%
}}}\circ TL\phi _{0}.
\end{eqnarray*}%
Now, let $\xi :TLB\rightarrow \overline{A}$ be some algebra morphism. Then $\xi $
coequalizes at the same time both pairs in $\left( \ref{eq:14bis}\right) $
if and only if%
\begin{eqnarray*}
\xi \circ TLm_{B} &=&\xi \circ \gamma _{TLB}\circ Tm_{\Omega TLB}\circ
T\left( \alpha _{LB}\otimes \alpha _{LB}\right) \circ T\epsilon _{LB\otimes
LB}\circ TL\phi _{2}\left( LB,LB\right) \circ TL\left( \eta _{B}\otimes \eta
_{B}\right) , \\
\xi \circ TLu_{B} &=&\xi \circ \gamma _{TLB}\circ Tu_{\Omega TLB}\circ
T\epsilon _{\unit _{\Aa  }}\circ TL\phi _{0}.
\end{eqnarray*}%
These are equalities in $\hom _{\Alg \left( \Aa  \right)
}\left( TL\left( B\otimes B\right) ,\overline{A}\right) $ and $\hom _{%
\Alg \left( \Aa  \right) }\left( TL\unit ,\overline{A}%
\right)$, respectively. Note that, using the adjunction $\left( T,\Omega
\right)$, one has that the map%
\begin{equation*}
\hom _{\Alg \left( \Aa  \right) }\left( TX,Y\right) %
\xrightarrow{\Phi \left( X,Y\right) }\hom _{\Aa  }\left(
X,\Omega Y\right) :f\mapsto \Omega f\circ \alpha _{X}
\end{equation*}%
has inverse%
\begin{equation*}
\hom _{\Aa  }\left( X,\Omega Y\right) \xrightarrow{\Phi \left(
X,Y\right) ^{-1}}\hom _{\Alg \left( \Aa  \right) }\left(
TX,Y\right) :g\mapsto \gamma _{Y}\circ Tg
\end{equation*}%
By applying $\Phi \left( X,Y\right) $, the equalities above reduce to%
\begin{eqnarray*}
\Omega \xi \circ \alpha _{LB}\circ Lm_{B} &=&\Omega \xi \circ m_{\Omega
TLB}\circ \left( \alpha _{LB}\otimes \alpha _{LB}\right) \circ \epsilon
_{LB\otimes LB}\circ L\phi _{2}\left( LB,LB\right) \circ L\left( \eta
_{B}\otimes \eta _{B}\right) , \\
\Omega \xi \circ \alpha _{LB}\circ Lu_{B} &=&\Omega \xi \circ u_{\Omega
TLB}\circ \epsilon _{\unit _{\Aa  }}\circ L\phi _{0}
\end{eqnarray*}%
i.e.%
\begin{eqnarray*}
\Omega \xi \circ \alpha _{LB}\circ Lm_{B} &=&\Omega \xi \circ m_{\Omega
TLB}\circ \left( \alpha _{LB}\otimes \alpha _{LB}\right) \circ \psi
_{2}\left( B,B\right) , \\
\Omega \xi \circ \alpha _{LB}\circ Lu_{B} &=&\Omega \xi \circ u_{\Omega
TLB}\circ \psi _{0}
\end{eqnarray*}%
Since $\xi :TLB\rightarrow \overline{A}$ is an algebra morphism, $\Omega \xi
\circ m_{\Omega TLB}=m_{A}\circ \left( \Omega \xi \otimes \Omega \xi \right)
$ and $\Omega \xi \circ u_{\Omega TLB}=u_{A}$ so that, if we set $q_\xi:=\Omega
\xi \circ \alpha _{LB}:LB\rightarrow A$, the last displayed equalities above
can be rewritten as 
\begin{eqnarray*}
q_\xi\circ Lm_{B} &=&m_{A}\circ \left( q_\xi\otimes q_\xi\right) \circ \psi _{2}\left(
B,B\right) , \\
q_\xi\circ Lu_{B} &=&u_{A}\circ \psi _{0}.
\end{eqnarray*}%
Since, from the very beginning, $\overline{A}=\left( A,m_{A},u_{A}\right) $
is an algebra, the last displayed equalities mean that $\left( \overline{A}%
,q_\xi\right) $ is an induced algebra of $\overline{B}.$ More precisely, an
algebra morphism $\xi :TLB\rightarrow \overline{A}$ coequalizes the pairs in \eqref{eq:14bis} if and only if $\left( \overline{A},q_\xi:LB\rightarrow A\right) $ is an induced algebra of $%
\overline{B}.$
By the foregoing, we have that $F$ is well-defined on objects. Moreover If $\overline{f}:(\overline{A},\xi) \to (\overline{A'},\xi')$ is a morphism in $\cocone(\overline{B})$, then $\Omega \overline{f}\circ q_\xi=\Omega \overline{f}\circ \Omega
\xi \circ \alpha _{LB}=\Omega (\overline{f}\circ
\xi) \circ \alpha _{LB}=\Omega\xi' \circ \alpha _{LB}=q_{\xi'}$ so that $\overline{f}:(\overline{A},q_\xi) \to (\overline{A'},q_{\xi'})$ is a morphism in $\indalg(\overline{B})$ and hence $F$ is well-defined on morphisms too.
Let now $(\overline{A},q)$ be an object in $\indalg(\overline{B})$. Via the adjunction $(T,\Omega,\alpha,\gamma)$, we have that $q=\Omega
\xi_q \circ \alpha _{LB}:LB\rightarrow A$ where $\xi_q:=\gamma_{\overline{A}}\circ Tq$. By the first part of the statement, we have that $\xi_q$ coequalizes \eqref{eq:14bis} so that $(\overline{A},\xi_q)$ is an object in $\cocone(\overline{B})$. Thus we can define $G:\indalg(\overline{B})\to \cocone(\overline{B}),\, (\overline{A},q)\mapsto (\overline{A},\xi_q),\,\overline{f}\mapsto \overline{f}$. Note that $G$ is well-defined on morphisms as, given $\overline{f}:(\overline{A},q)\to (\overline{A'},q')$, we have $\overline{f}\circ\xi_q=\overline{f}\circ\gamma_{\overline{A}}\circ Tq=\gamma_{\overline{A'}}\circ T\Omega\overline{f}\circ Tq=\gamma_{\overline{A'}}\circ T(\Omega\overline{f}\circ q)=\gamma_{\overline{A'}}\circ Tq'=\xi_{q'}$. Since $(T,\Omega,\alpha,\gamma)$ is an adjunction, it is clear that $\xi=\xi_{q_\xi}$ and $q=q_{\xi_q}$ so that we get that $F\circ G$ and $G\circ F$ act as the identity functors on objects. Since they also act as the identity on morphisms, we get $F\circ G=\Id$ and $G\circ F=\Id$.
\end{proof}

As a consequence of Proposition \ref{pro:induced1} and of the observation we made that a colimit for \eqref{eq:14bis} is nothing but an initial object in the category $\cocone(\overline{B})$, we get the following characterization.

\begin{proposition}
\label{pro:induced2} In the Setting \ref{not:setting}, the
following assertions are equivalent for any algebra $\overline{B}:=\left(
B,m_{B},u_{B}\right) $ in $\Bb  $.

\begin{enumerate}
\item $\left( \overline{P},p:LB\rightarrow
\Omega \overline{P}\right) $ is an initial object in the category $\indalg(\overline{B})$ of induced algebras of $\overline{B}$. 

\item $\left( \overline{P},\kappa:TLB\rightarrow
\overline{P} \right) $ is a colimit for \eqref{eq:14bis}.
\end{enumerate}

The morphisms $p$ and $\kappa $ correspond to each other through the adjunction
$(T,\Omega,\alpha,\gamma)$ i.e. $p:=\Omega \kappa \circ \alpha _{LB}$ and $%
\kappa :=\gamma _{P}\circ Tp$.
\end{proposition}

By using Proposition \ref{pro:induced2} we are now able to rewrite Theorem \ref{teo:Tambara} in a different form.

\begin{theorem}
\label{teo:redinv}In the Setting \ref{not:setting}, assume
that for any algebra $\overline{B}:=\left( B,m_{B},u_{B}\right) $ in $%
\Bb  $ there is an initial object $\left( \overline{P}_{\overline{B}%
},p_{\overline{B}}:LB\rightarrow \Omega \overline{P}_{\overline{B}}\right) $
in the category $\indalg(\overline{B})$ of induced algebras of $\overline{B}$.

Then $\overline{R}=\Alg(R):\Alg (\Aa  )\rightarrow \Alg (%
\Bb  )$ has a left adjoint $\overline{L}$ defined by $\overline{L}\,\overline{B}:=\overline{P}_{\overline{B}}$ for any $\overline{B}$ as above. Moreover the morphisms $p_{\overline{B}}$ define a natural transformation $p:L\Omega
\rightarrow \Omega \overline{L}$ whose naturality completely determines how $\overline{L}$ acts on morphisms.
\end{theorem}

\begin{proof}
Since condition $(1)$ in Proposition \ref{pro:induced2} is satisfied, we know there is a morphism $\kappa _{\overline{B}}:TLB\rightarrow
\overline{P}_{\overline{B}}$ such that $\left( \overline{P}_{\overline{B}},\kappa_{\overline{B}} \right) $ is a colimit for $\left( \ref{eq:14bis}\right) .$ Moreover we have that $p_{\overline{B}}=\Omega\kappa_{\overline{B}}\circ\alpha_{LB}$.

 By Theorem \ref{teo:Tambara}, this colimit yields a
functor $\overline{L}$ which is a left adjoint of the functor $\overline{R}:{%
\Alg }({\Aa  })\rightarrow {\Alg }({\Bb  })$ and the morphisms $\kappa_{\overline{B}}$ define a natural transformation $\kappa:TL\Omega\to \overline{L}$. Explicitly $\overline{L}\,\overline{B}:=\overline{P}_{\overline{B}}$ and the action of $\overline{L}$ on morphisms is uniquely determined by the naturality of $\kappa$. From $p_{\overline{B}}=\Omega\kappa_{\overline{B}}\circ\alpha_{LB}$ we get that the morphisms $p_{\overline{B}}$ define a natural transformation $p:L\Omega
\rightarrow \Omega \overline{L}$ such that $p=\Omega\kappa\circ\alpha L\Omega$.
Indeed, since we also have $\kappa_{\overline{B}} =\gamma _{\overline{L}\,\overline{B}}\circ Tp_{\overline{B}}$, the naturality of $\kappa$ is equivalent to the naturality of $p$ and hence the latter completely determines the action of $\overline{L}$ on morphisms as well.
\end{proof}

%

Theorem \ref{teo:redinv} shows how central is the role played by an initial object in the category $\indalg(\overline{B})$ of induced algebras of $\overline{B}$. Under suitable assumptions, we will see that this category can be replaced by three other categories that can be more easy to handle in practice. One of them is $\indobj(\overline{B})$ while the remaining two, namely $\epindobj(\overline{B})$ and $\epindalg(\overline{B})$, are introduced in Subsection \ref{sec:epi}. 

\subsection{Epi-induced objects and algebras and initial objects}\label{sec:epi}

Here we introduce the categories $\epindobj(\overline{B})$ and $\epindalg(\overline{B})$ and, in Theorem \ref{teo:red}, we show that, under the proper assumptions, the four categories $\epindobj(\overline{B})$, $\indobj(\overline{B})$, $\epindalg(\overline{B})$ and $\indalg(\overline{B})$ have the same initial object, if any. This will be done by exploiting the results  on relative weak coreflections and fibrations of Section \ref{sec:relweakcoref}.

\begin{definition}
By an \textit{epi-induced object (or algebra) of }$%
\overline{B}$ we mean an induced object (or algebra) $\left( \overline{A},q\right) $ of $%
\overline{B}$ such that $q$ is an epimorphism.

We denote by $\epindobj(\overline{B})$ the full subcategory of $\indobj(\overline{B})$ formed by epi-induced objects and by $\epindalg(\overline{B})$ the full subcategory of $\indalg(\overline{B})$ formed by epi-induced algebras.
\end{definition}

\begin{remark}\label{rem:functors}
Clearly the forgetful functor $\Omega :\Alg \left(
\Aa  \right) \rightarrow \Aa  $ induces the faithful functors
\begin{gather*}
U: \epindalg(\overline{B})\to\epindobj(\overline{B}),\quad  (\overline{A},q)\mapsto(\overline{A},q),\quad \overline{h}\mapsto \Omega \overline{h},\\V: \indalg(\overline{B})\to\indobj(\overline{B}),\quad  (\overline{A},q)\mapsto(\overline{A},q),\quad \overline{h}\mapsto \Omega \overline{h},
\end{gather*}
that make commute the following diagram of functors
\begin{equation*}
\xymatrixrowsep{.6cm}\xymatrix{ \epindalg(\overline{B})\pulb\ar@{^(->}[d] \ar[r]^-{U} &\epindobj(\overline{B})\ar@{^(->}[d] \\\indalg(\overline{B}) \ar[r]^-{V} &\indobj(\overline{B})}
\end{equation*}
where the vertical arrows  are the canonical full embeddings. Note that the category $\epindalg(\overline{B})$ is the pullback of $\epindobj(\overline{B})$ along $V$  meaning that it is the full subcategory of $\indalg(\overline{B})$ consisting of objects whose image through $V$ belongs to $\epindobj(\overline{B})$.
\end{remark}

The assumptions we will use, include the notion of (Epi, StrongMono)-factorization. Let us
recall the definition of a strong monomorphism in a
category.

\begin{definition}
A monomorphism $m $ is called \textit{strong} if for every commutative square%
\begin{equation*}
\begin{array}{ccc}
\xymatrix{\bullet\ar[d]_ {u}\ar[rr]^-{e} && {\bullet}\ar[d]^{v}
\ar@{.>}[dll]_w \\ \bullet \ar[rr]^-{m} && {\bullet} } &  &
\end{array}%
\end{equation*}
where $e$ is an epimorphism, there is a unique morphism $w$ such that $%
w\circ e=u $ and $m\circ w=v.$ 
\end{definition}

For instance, one can easily verify that a regular monomorphism is always
strong.

\begin{remark}\label{rem:coimage}
Following \cite[page 12]{Mitchell}, recall that a \textit{coimage} of a
morphism $q:Q\rightarrow A$ in an arbitrary category $\Cc$ is a pair $\mathrm{Coim}%
\left( q\right) :=\left( A^{\prime },q^{\prime }\right) $ where $q^{\prime }:Q\rightarrow A^{\prime
}$ is an epimorphism such that $q$ factors through $q'$ and, if there is another epimorphism $q^{\prime
\prime }:Q\rightarrow A^{\prime \prime }$ such that $q$ factors through $q''$, then $q'$ factors through $q''$. In other words $\mathrm{Coim}\left(q\right) $ is the biggest epi-induced object of $Q$ that $q$ factors through.
\begin{equation*}
\xymatrixrowsep{.5cm}\xymatrix{&&Q\ar[lld]_{q''}\ar[ld]^{q'}\ar[dd]^{q}\\A''\ar@{.>}[r]\ar@{.>}[drr]&A'\ar@{.>}[dr]\\&&A}
\end{equation*}
Now, consider a morphism $q:Q\rightarrow A$ in $\Cc$ that admits an (Epi, StrongMono)-factorization
i.e. $q$ factors as an epimorphism $q^{\prime }:Q\rightarrow A^{\prime }$
followed by a strong monomorphism $h^{\prime }:A^{\prime }\rightarrow A,$ so
that $q=h^{\prime }\circ q^{\prime }.$ Then $\left( A^{\prime },q^{\prime
}\right) =\mathrm{Coim}\left( q\right) .$
\end{remark}

\subsection{Comparing the Initial objects in \texorpdfstring{$\indobj(\overline{B})$}{TEXT} and \texorpdfstring{$\epindobj(\overline{B})$}{TEXT}}
We are going to prove Proposition \ref{prop:initialepiobj} which compares the initial objects in $\indobj(\overline{B})$ and $\epindobj(\overline{B})$. First we need two lemmata.

\begin{lemma}
\label{lem:refine}Let $(L:\Bb  \to\Aa  ,\psi_2,\psi_0)$ be a
colax monoidal functor and let $\left( \overline{E},q\right) \in\indobj(\overline{B})$ be such that

\begin{itemize}
\item $q$ factors as $q=h\circ q^{\prime }$, where $h:E^{\prime
}\rightarrow E$ is a strong monomorphism and $q^{\prime }:LB\rightarrow
E^{\prime }$ is an epimorphism;

\item the morphisms $\left( q^{\prime }\otimes q^{\prime }\right) \circ \psi
_{2}\left( B,B\right) $ and $\psi _{0}$ are epimorphisms.
\end{itemize}
Then $E^{\prime }$ becomes an $\left( \overline{E^{\prime }},q^{\prime }\right) \in\indobj(\overline{B})$ and $h$ induces a morphism $h:(\overline{E^{\prime }},q^{\prime })\to (\overline{E},q)$ in $\indobj(\overline{B})$ such that $h\circ m_{E^{\prime }}=m_{E}\circ \left( h\otimes h\right) $ and $h\circ u_{E^{\prime }}=u_{E}$.
\end{lemma}

\begin{proof}
We have%
\begin{equation*}
m_{E}\circ \left( h\otimes h\right) \circ \left( q^{\prime }\otimes
q^{\prime }\right) \circ \psi _{2}\left( B,B\right) =m_{E}\circ \left(
q\otimes q\right) \circ \psi _{2}\left( B,B\right) =q\circ Lm_{B}=h\circ
q^{\prime }\circ Lm_{B}
\end{equation*}%
and $u_{E}\circ \psi _{0}=q\circ Lu_{B}=h\circ q^{\prime }\circ Lu_{B}.$
Hence we have the following commutative squares
\begin{equation*}
\begin{array}{ccc}
\xymatrix{L\left( B\otimes B\right) \ar[d]_ {q^{\prime }\circ
Lm_{B}}\ar[rr]^-{\left( q^{\prime }\otimes q^{\prime}\right) \circ \psi
_{2}\left( B,B\right) } && E^{\prime }\otimes E^{\prime }\ar[d] ^{m_{E}\circ
\left( h\otimes h\right)}\ar@{.>}[dll]_{m_{E'}} \\ E^{\prime } \ar[rr]^-{h}
&& {E} } &  &
\end{array}
\qquad
\begin{array}{ccc}
\xymatrix{L\unit \ar[d]_{q^{\prime }\circ Lu_{B}}\ar[rr]^-{\psi _{0}} &&
{\unit }\ar[d]^{u_{E}} \ar@{.>}[dll]_{u_{E'}} \\ E^{\prime }
\ar[rr]^-{h} && {E} } &  &
\end{array}%
\end{equation*}
Since the morphisms $\left( q^{\prime }\otimes q^{\prime }\right) \circ \psi
_{2}\left( B,B\right) $ and $\psi _{0}$ are epimorphisms, and $h$ is a
strong monomorphism, there is a unique morphism $m_{E^{\prime }}$ such that $%
h\circ m_{E^{\prime }}=m_{E}\circ \left( h\otimes h\right) $ and $%
m_{E^{\prime }}\circ \left( q^{\prime }\otimes q^{\prime }\right) \circ \psi
_{2}\left( B,B\right) =q^{\prime }\circ Lm_{B},$ and there is a unique
morphism $u_{E^{\prime }}$ such that $h\circ u_{E^{\prime }}=u_{E}$ and $%
u_{E^{\prime }}\circ \psi _{0}=q^{\prime }\circ Lu_{B}.$ Thus $\left(
\overline{E^{\prime }},q^{\prime }\right) \in\indobj(\overline{B})$, where we set $\overline{E^{\prime }}:=\left( E^{\prime
},m_{E^{\prime }},u_{E^{\prime }}\right) $.
\end{proof}

\begin{lemma}\label{lem:posetalobj}$\epindobj(\overline{B})$ is a replete posetal full subcategory of $\indobj(\overline{B})$.
 \end{lemma}

\begin{proof}
Let $(\overline{A},p)\in\epindobj(\overline{B})$ and $(\overline{A'},p')\in\indobj(\overline{B})$.
\begin{itemize}
  \item Given an isomorphism $h:(\overline{A},p)\to (\overline{A'},p')$, we have $p'=h\circ p$ so that $p'$ is an epimorphism as $p$. Thus $\epindobj(\overline{B})$ is a replete full subcategory of $\indobj(\overline{B})$.
  \item Given morphisms $f,g:(\overline{A},p)\to (\overline{A'},p')$ in $\indobj(\overline{B})$, we get $f\circ p=p'= g\circ p$. Since $p$ is an epimorphism we get $f=g$. In particular $\epindobj(\overline{B})$ is posetal.\qedhere
\end{itemize}
\end{proof}

Denote by $\Ee(\overline{B})$ the class of morphisms $h:(\overline{A^{\prime }},q^{\prime })\to (\overline{A},q)$ in $\indobj(\overline{B})$ such that $h:A'\to A $ is a monomorphism, $h\circ m_{A^{\prime }}=m_{A}\circ \left( h\otimes h\right) $ and $h\circ u_{A^{\prime }}=u_{A}$.

\begin{proposition}\label{prop:initialepiobj}
 In the Setting \ref{not:setting}, assume that

\begin{itemize}
\item If $\left( \overline{A},q\right) \in \indobj(\overline{B})$, then $q$ admits an (Epi,
StrongMono)-factorization $q=h\circ q^{\prime }$,

\item the morphisms $\left( q^{\prime }\otimes q^{\prime }\right) \circ \psi
_{2}\left( B,B\right) $ and $\psi _{0}$ are epimorphisms.
\end{itemize}
Then, $\epindobj(\overline{B})$ is a weakly $\Ee(\overline{B})$-coreflective subcategory of $\indobj(\overline{B})$. Explicitly, given $\left( \overline{A},q\right) \in \indobj(\overline{B})$, we have that $( \overline{A},q)^\star=(\overline{A'},q')$ where $\left( A',q'\right)=\mathrm{Coim}(q)$.
As a consequence $\indobj(\overline{B})$ and $\epindobj(\overline{B})$ have the same initial objects.
\end{proposition}

\begin{proof}
Let $\left( \overline{A},q:LB\rightarrow A\right) $ be an induced object of
$\overline{B}$ in $\Aa  .$ By hypothesis, $q$ admits an (Epi,
StrongMono)-factorization $q=h\circ q^{\prime }$ where $h:A^{\prime
}\rightarrow A$ is a strong monomorphism and $q^{\prime
}:LB\rightarrow A^{\prime }$ is an epimorphism. Moreover the morphisms $%
\left( q^{\prime }\otimes q^{\prime }\right) \circ \psi _{2}\left(
B,B\right) $ and $\psi _{0}$ are epimorphisms. Note that, by Remark \ref{rem:coimage}, we have that $\left( A',q'\right)=\mathrm{Coim}(q)$. We can apply Lemma \ref{lem:refine} to deduce that $A^{\prime }$ becomes an epi-induced object $\left( \overline{A^{\prime }},q^{\prime }\right) $ of $\overline{B}$ and $h$ induces a morphism $h:(\overline{A^{\prime }},q^{\prime })\to (\overline{A},q)$ of induced objects such that $h\circ m_{A^{\prime }}=m_{A}\circ \left( h\otimes h\right) $ and $h\circ u_{A^{\prime }}=u_{A}$. Thus $h:(\overline{A^{\prime }},q^{\prime })\to (\overline{A},q)$ is in $\Ee(\overline{B})$.

We have so proved that, for any $\left( \overline{A},q\right) $ in $\indobj(\overline{B})$, there is $\left(
\overline{A^{\prime }},q^{\prime }\right) $ in $\epindobj(\overline{B})$ and a morphism  $h:(\overline{A^{\prime }},q^{\prime })\to (\overline{A},q)$ in $\Ee(\overline{B})$.
 Given $(\overline{E},p)$ in $\epindalg(\overline{B})$, let us check that $$\hom_{\indobj(\overline{B})}((\overline{E},p),h):\hom_{\indobj(\overline{B})}((\overline{E},p),(\overline{A^{\prime }},q^{\prime }))\to \hom_{\indobj(\overline{B})}((\overline{E},p),(\overline{A},q))$$ is surjective. Given a morphism $f:(\overline{E},p)\to(\overline{A},q)$ in $\indobj(\overline{B})$, we have that $f\circ p=q=h\circ q'$ so that $f\circ p=h\circ q'$. Since $h$ is a strong monomorphism and $p$ is an epimorphism, there is a unique morphism $w:E\to A$ such that $h\circ w=f$ and $w\circ p=q'$. These equalities say we have a morphism $w:(\overline{E},p)\to(\overline{A^{\prime }},q^{\prime })$ whose image through $\hom_{\indalg(\overline{B})}((\overline{E},p),h)$ is exactly the starting morphism $f$.
Thus $\hom_{\indalg(\overline{B})}((\overline{E},p),h)$ is surjective. Hence $\epindalg(\overline{B})$ is a weakly $\Ee(\overline{B})$-coreflective subcategory of $\indalg(\overline{B})$. In particular $\epindalg(\overline{B})$ is a weakly coreflective subcategory of $\indalg(\overline{B})$. This, together with Lemma \ref{lem:posetalobj}, implies that we can apply Proposition \ref{pro:initial} to conclude.
\end{proof}

\subsection{Comparing the Initial objects in \texorpdfstring{$\indalg(\overline{B})$}{TEXT} and \texorpdfstring{$\epindalg(\overline{B})$}{TEXT}}
Next aim is proving Proposition \ref{prop:initialepi}, which compares the initial objects of $\indalg(\overline{B})$ and $\epindalg(\overline{B})$. We first need the following lemmata.

\begin{lemma}\label{lem:posetalalg}$\epindalg(\overline{B})$ is a replete posetal full subcategory of $\indalg(\overline{B})$.
 \end{lemma}

\begin{proof}
It follows by Remark \ref{rem:functors} and Lemma \ref{lem:pullback}.
\end{proof}

\begin{lemma}\label{lem:Vfib}
The functor $V:\indalg(\overline{B}) \to \indobj(\overline{B})$ of Remark \ref{rem:functors} is an $\Ee(\overline{B})$-fibration.
\end{lemma}

\begin{proof}
Let $\left( \overline{A},q\right) \in \indalg(\overline{B})$, $\left(
\overline{A^{\prime }},q^{\prime }\right)\in\indobj(\overline{B})$ and let $h:(\overline{A^{\prime }},q^{\prime })\to V(\overline{A},q)$ be a morphism in $\Ee(\overline{B}).$ Thus $h:A'\to A$ is a monomorphism such that $h\circ m_{A^{\prime }}=m_{A}\circ \left( h\otimes h\right) $ and $h\circ u_{A^{\prime }}=u_{A}$. Since $h$ is a monomorphism, one easily checks that $\overline{A^{\prime }}=(A^{\prime },m_{A^{\prime }},u_{A^{\prime }})$ is an algebra, by using the fact that $\overline{A}=(A,m_{A},u_{A})$ is an algebra. Then $h$ induces an algebra morphism $\overline{h}:\overline{A^{\prime }}\to \overline{A}$ such that $\Omega \overline{h}=h$ and  we get a morphism $\overline{h}:(\overline{A^{\prime }},q^{\prime })\to (\overline{A},q)$ in $\indalg(\overline{B})$ whose image through $V$ is $h:(\overline{A^{\prime }},q^{\prime })\to V(\overline{A},q)$. It remains to check that $\overline{h}$ is cartesian over $h$. Given a morphism $\overline{g}:(\overline{A''},q')\to (\overline{A},q)$ in $\indalg(\overline{B})$ and a morphism $l:V(\overline{A''},q'')\to V(\overline{A^{\prime }},q^{\prime })$ such that $h\circ l=V\overline{g}=:g$, we have $h\circ l\circ m_{A''}=g\circ m_{A''}=m_{A}\circ (g\otimes g)=m_{A}\circ (h\otimes h)\circ (l\otimes l)=h\circ m_{A'}\circ (l\otimes l)$ so that $l\circ m_{A''}= m_{A'}\circ (l\otimes l)$ as $h$ is a monomorphism. Similarly $h\circ l\circ u_{A''}=g\circ u_{A''}= u_{A}=h\circ u_{A'}$ and hence $l\circ u_{A''}=u_{A'}$. Therefore there is an algebra morphism $\overline{l}:\overline{A''}\to \overline{A'}$ such that $\Omega \overline{l}=l$. Thus $\overline{l}:(\overline{A''},q'')\to (\overline{A^{\prime }},q^{\prime })$ is a morphism whose image through $V$ is $l:V(\overline{A''},q'')\to V(\overline{A^{\prime }},q^{\prime })$ and such that $\overline{h}\circ \overline{l}=\overline{g}$.
\end{proof}

\begin{proposition}\label{prop:initialepi}
 In the Setting \ref{not:setting}, assume that

\begin{itemize}
\item If $\left( \overline{A},q\right) \in \indobj(\overline{B})$, then $q$ admits an (Epi,
StrongMono)-factorization $q=h\circ q^{\prime }$,

\item the morphisms $\left( q^{\prime }\otimes q^{\prime }\right) \circ \psi
_{2}\left( B,B\right) $ and $\psi _{0}$ are epimorphisms.
\end{itemize}
Then, $\epindalg(\overline{B})$ is a weakly coreflective subcategory of $\indalg(\overline{B})$. As a consequence $\indalg(\overline{B})$ and $\epindalg(\overline{B})$ have the same initial objects.
\end{proposition}

\begin{proof}
Our hypotheses guarantee that we can apply Proposition \ref{prop:initialepiobj} to get that $\epindobj(\overline{B})$ is a weakly $\Ee(\overline{B})$-coreflective subcategory of $\indobj(\overline{B})$.  Moreover, by Lemma \ref{lem:Vfib}, the functor  $V:\indalg(\overline{B}) \to \indobj(\overline{B})$ is an $\Ee(\overline{B})$-fibration. Therefore, we can apply Proposition \ref{pro:pullbackweak} to the diagram in Remark \ref{rem:functors} to get that $\epindalg(\overline{B})$ is a weakly coreflective subcategory of $\indalg(\overline{B})$. This, together with Lemma \ref{lem:posetalalg} imply that we can apply Proposition \ref{pro:initial} to conclude.
\end{proof}

\subsection{Comparing all of the Initial objects}

Next aim is to obtain Theorem \ref{teo:red}, where we compare the initial objects in the categories $\epindobj(\overline{B})$, $\indobj(\overline{B})$, $\epindalg(\overline{B})$ and $\indalg(\overline{B})$ altogether. First we need some lemmata.\medskip

Given an induced algebra $\left( \overline{E},q\right) $ of $\overline{B},$
the next lemma shows that, under mild assumptions, $\mathrm{Coim}\left(
q\right) =\left( E^{\prime },q^{\prime }\right) $ becomes an induced algebra
of $\overline{B}$. 

\begin{lemma}
\label{lem:red}Let $\left( L:\Bb  \rightarrow \Aa  ,\psi
_{2},\psi _{0}\right) $ be a colax monoidal functor. Let $q:LB\rightarrow A$
me a morphism that admits two (Epi, StrongMono)-factorizations $q=h\circ
q^{\prime }$ and $q=h^{\prime }\circ q^{\prime \prime }$. We have that
\begin{itemize}
  \item $\left( q^{\prime }\otimes q^{\prime }\right)
\circ \psi _{2}\left( B,B\right) $ is an
epimorphism if and only if so is $\left( q^{\prime \prime }\otimes q^{\prime \prime
}\right) \circ \psi _{2}\left( B,B\right) $;
  \item $\left( q^{\prime }\otimes
q^{\prime }\otimes q^{\prime }\right) \circ \left( LB\otimes \psi _{2}\left(
B,B\right) \right) \circ \psi _{2}\left( B,B\otimes B\right) $ is an
epimorphism if and only if  so is $\left( q^{\prime \prime
}\otimes q^{\prime \prime }\otimes q^{\prime \prime }\right) \circ \left(
LB\otimes \psi _{2}\left( B,B\right) \right) \circ \psi _{2}\left(
B,B\otimes B\right) $.
\end{itemize}
\end{lemma}

\begin{proof}
Denote by $P^{\prime }$ the domain of $h$ and by $P^{\prime \prime }$ the
domain of $h^{\prime }$. By uniqueness of the (Epi,
StrongMono)-factorizations, we have an isomorphism $w:P^{\prime }\rightarrow
P^{\prime \prime }$ such that $w\circ q^{\prime }=q^{\prime \prime }.$ Hence
$\left( q^{\prime \prime }\otimes q^{\prime \prime }\right) \circ \psi
_{2}\left( B,B\right) =\left( w\otimes w\right) \circ \left( q^{\prime
}\otimes q^{\prime }\right) \circ \psi _{2}\left( B,B\right) $ from which the first item follows. Similarly one treats the second one.
\end{proof}

\begin{lemma}\label{lem:red2}
In the Setting \ref{not:setting}, assume that $\psi_0$ is an epimorphism and let $\left( \overline{A},q\right) \in \epindalg(\overline{B})$ be such that $q$ admits an (Epi,StrongMono)-factorization $q=h\circ q^{\prime }$.
\begin{enumerate}
  \item[$1)$] If $\left( q^{\prime }\otimes q^{\prime }\right) \circ \psi _{2}\left( B,B\right) $ is an epimorphism, then so is  $\left( q\otimes q\right) \circ \psi _{2}\left( B,B\right) $.
  \item[$2)$] If $\left( q^{\prime }\otimes
q^{\prime }\otimes q^{\prime }\right) \circ \left( LB\otimes \psi _{2}\left(
B,B\right) \right) \circ \psi _{2}\left( B,B\otimes B\right) $ is an epimorphism, then so is $\left( q\otimes
q\otimes q\right) \circ \left( LB\otimes \psi _{2}\left(
B,B\right) \right) \circ \psi _{2}\left( B,B\otimes B\right) $.
\end{enumerate}
\end{lemma}

\begin{proof}
 We just prove $1)$, the argument for $2)$ being similar. Since $q $ is an epimorphism, then $q=\mathrm{Id}\circ q$ is (Epi, StrongMono)-factorization. Since $\left( q^{\prime }\otimes q^{\prime }\right) \circ \psi _{2}\left( B,B\right) $ is an epimorphism, by Lemma \ref{lem:red}, so is $\left( q\otimes q\right) \circ \psi
_{2}\left( B,B\right) $.
\end{proof}

\begin{lemma}
\label{lem:indalgmor2} Let $\left( L:\Bb  \rightarrow \Aa  ,\psi
_{2},\psi _{0}\right) $ be a colax monoidal functor and let $\overline{B}%
\in\Alg(\Bb ) $. Let $\left(
\overline{A},q\right) $ and $\left( \overline{A^{\prime }}%
,q^{\prime }\right) $ be in $\indalg(\overline{B})$. Assume that $\left( q\otimes q\right) \circ \psi _{2}\left(
B,B\right) $ and $\psi _{0}$ are epimorphisms. Then any morphism $h:A\rightarrow A^{\prime }$ such that $h\circ q=q^{\prime
} $ becomes a morphism $\overline{h}:(\overline{A},q)\to (\overline{A'},q') $ in $\indalg(\overline{B})$.
\end{lemma}

\begin{proof}
We compute%
\begin{eqnarray*}
m_{A^{\prime }}\circ \left( h\otimes h\right) \circ \left( q\otimes q\right)
\circ \psi _{2}\left( B,B\right) &=&m_{A^{\prime }}\circ \left( q^{\prime
}\otimes q^{\prime }\right) \circ \psi _{2}\left( B,B\right) =q^{\prime
}\circ Lm_{B} \\
&=&h\circ q\circ Lm_{B}=h\circ m_{A}\circ \left( q\otimes q\right) \circ
\psi _{2}\left( B,B\right) \\
h\circ u_{A}\circ \psi _{0} &=&h\circ q\circ Lu_{B}=q^{\prime }\circ
Lu_{B}=u_{A^{\prime }}\circ \psi _{0}
\end{eqnarray*}%
so that, in view of the assumptions, we deduce that $m_{A^{\prime }}\circ
\left( h\otimes h\right) =h\circ m_{A}$ and $h\circ u_{A}=u_{A^{\prime }}$
i.e. that $h$ becomes an algebra morphism $\overline{h}:\overline{A}\rightarrow \overline{%
A^{\prime }}$ such that $\Omega \overline{h}=h$. Since $\Omega \overline{h}\circ q=q^{\prime
} $ we get that $\overline{h}$ is a morphism in $\indalg(\overline{B})$.
\end{proof}

\begin{lemma}\label{lem:Veff}
In the Setting \ref{not:setting}, assume that

\begin{itemize}
\item If $\left( \overline{A},q\right) \in \epindalg(\overline{B})$, then $q$ admits an (Epi,
StrongMono)-factorization $q=h\circ q^{\prime }$,

\item the morphisms $\left( q^{\prime }\otimes q^{\prime }\right) \circ \psi
_{2}\left( B,B\right) $ and $\psi _{0}$ are epimorphisms.
\end{itemize} Then the functor $U:\epindalg(\overline{B})\to \epindobj(\overline{B})$ of Remark \ref{rem:functors} is fully faithful.
\end{lemma}

\begin{proof}
 By construction $U$ is faithful. Let $(\overline{A},q),(\overline{A'},q')\in\epindalg(\overline{B})$ and let $h:(\overline{A},q)\to(\overline{A'},p)$ be a morphism in $\epindobj(\overline{B})$. Then  $h\circ q=p$. By Lemma \ref{lem:red2} $1)$, we have that $\left( q\otimes q\right) \circ \psi
_{2}\left( B,B\right) $ is an epimorphism  so that we can apply Lemma \ref{lem:indalgmor2} to get an algebra morphism $\overline{h}$ such that $\Omega \overline{h}=h$. Therefore $\Omega \overline{h}\circ q=p$ and hence we have a morphism $\overline{h}:(\overline{A},q)\to(\overline{A'},p)$ in $\epindalg(\overline{B})$ whose image through $U$ is $h$.
\end{proof}

The next aim is to reduce to the case where epi-induced object are epi-induced
algebras.

\begin{lemma}
\label{lem:bast} Let $\left( L:\Bb  \rightarrow \Aa  ,\psi
_{2},\psi _{0}\right) $ be a colax monoidal functor and let $\overline{B}%
\in\Alg(\Bb ) $.

Let $\left(
\overline{A},q\right)\in \epindobj(\overline{B})$ be such that $\left( q\otimes q\otimes q\right) \circ \left( LB\otimes
\psi _{2}\left( B,B\right) \right) \circ \psi _{2}\left( B,B\otimes B\right)
$ is an epimorphism. Then $\left( \overline{A},q\right)\in\epindalg(\overline{B}) $.
\end{lemma}

\begin{proof}
Let $\left( \overline{E},q:LB\rightarrow E\right) \in \epindobj(\overline{B})$. One easily verifies that
\begin{equation*}
m_{E}\circ \left( m_{E}\otimes E\right) \circ \left( q\otimes q\otimes
q\right) \circ \left( \psi _{2}\left( B,B\right) \otimes LB\right) \circ
\psi _{2}\left( B\otimes B,B\right)=q\circ Lm_{B}\circ L\left( m_{B}\otimes
B\right)~\text{and}
\end{equation*}
\begin{equation*}
m_{E}\circ \left( E\otimes m_{E}\right) \circ \left( q\otimes q\otimes
q\right) \circ \left( LB\otimes \psi _{2}\left( B,B\right) \right) \circ
\psi _{2}\left( B,B\otimes B\right) =q\circ Lm_{B}\circ L\left( B\otimes
m_{B}\right) .
\end{equation*}

Since $m_{B}$ is associative and $\left( \psi _{2}\left( B,B\right) \otimes
LB\right) \circ \psi _{2}\left( B\otimes B,B\right) =\left( LB\otimes \psi
_{2}\left( B,B\right) \right) \circ \psi _{2}\left( B,B\otimes B\right) $
and $\left( q\otimes q\otimes q\right) \circ \left( LB\otimes \psi
_{2}\left( B,B\right) \right) \circ \psi _{2}\left( B,B\otimes B\right)$ is
an epimorphism, we deduce that $m_{E}$ is associative too. Note that%
\begin{eqnarray*}
l_{E}\circ \left( \psi _{0}\otimes q\right) \circ \psi _{2}\left( \unit %
,B\right) &=&l_{E}\circ \left( \unit \otimes q\right) \circ \left( \psi
_{0}\otimes LB\right) \circ \psi _{2}\left( \unit ,B\right) \\
&=&q\circ l_{LB}\circ \left( \psi _{0}\otimes LB\right) \circ \psi
_{2}\left( \unit ,B\right) =q\circ Ll_{B}
\end{eqnarray*}%
and hence, since $q$ is an epimorphism, we deduce that $\left( \psi
_{0}\otimes q\right) \circ \psi _{2}\left( \unit ,B\right) $ is an
epimorphism too. Using naturality of $\psi _{2}$, we have
\begin{eqnarray*}
m_{E}\circ \left( u_{E}\otimes E\right) \circ \left( \psi _{0}\otimes
q\right) \circ \psi _{2}\left( \unit ,B\right) &=&m_{E}\circ \left(
q\otimes q\right) \circ \left( Lu_{B}\otimes LB\right) \circ \psi _{2}\left(
\unit ,B\right) \\
&=&q\circ Lm_{B}\circ L\left( u_{B}\otimes B\right) \\
&=&q\circ L\left( l_{B}\right) \\
&=&q\circ l_{LB}\circ \left( \psi _{0}\otimes LB\right) \circ \psi
_{2}\left( \unit ,B\right) \\
&=&l_{E}\circ \left( \psi _{0}\otimes q\right) \circ \psi _{2}\left( \mathbf{%
1},B\right)
\end{eqnarray*}

so that $m_{E}\circ \left( u_{E}\otimes E\right) =l_{E}.$ Similarly one
proves that $m_{E}\circ \left( E\otimes u_{E}\right) =r_{E}$.
Then $\overline{E}$ is an algebra so that $\left( \overline{E},q\right)\in\epindalg(\overline{B}) $.
\end{proof}

We are now able to prove the announced result.


\begin{theorem}\label{teo:red} In the Setting \ref{not:setting}, assume that

\begin{itemize}
\item if $\left( \overline{A},q\right) \in\indobj(\overline{B})$, then $q$ admits an (Epi,
StrongMono)-factorization $q=h\circ q^{\prime }$,

\item the morphisms $\left( q^{\prime }\otimes q^{\prime }\right) \circ \psi
_{2}\left( B,B\right) $ and $\psi _{0}$ are epimorphisms,

\item the morphism $\left( q^{\prime }\otimes q^{\prime }\otimes
q^{\prime }\right) \circ \left( LB\otimes \psi _{2}\left( B,B\right) \right)
\circ \psi _{2}\left( B,B\otimes B\right) $ is an epimorphism.
\end{itemize}

The following assertions are equivalent.

\begin{enumerate}
\item[(1)] $\left( \overline{P},p\right) $ is an initial object in $\indobj(\overline{B})$.

\item[(2)] $\left( \overline{P},p\right) $ is an initial object in $\epindobj(\overline{B})$.

\item[(3)] $\left( \overline{P},p\right) $ is an initial object in $\indalg(\overline{B})$.

\item[(4)] $\left( \overline{P},p\right) $ is an initial object in $\epindalg(\overline{B})$.
\end{enumerate}
\end{theorem}

\begin{proof}
$\left( 1\right) \Leftrightarrow\left( 2\right).$ This follows from is Proposition \ref{prop:initialepiobj}.

$\left( 3\right) \Leftrightarrow\left( 4\right).$ This follows from is Proposition \ref{prop:initialepi}.

$\left( 2\right) \Leftrightarrow\left( 4\right) .$  By Lemma \ref{lem:Veff}, the functor $U:\epindalg(\overline{B})\to \epindobj(\overline{B})$ of Remark \ref{rem:functors} is fully faithful. By construction $U$ is also injective on objects. In order to conclude we check that it is also surjective on objects whence a category isomorphism. Let $(\overline{A},q)\in \epindobj(\overline{B})$.
By Lemma \ref{lem:red2} and the assumptions, we have that $\left( q\otimes q\otimes q\right) \circ
\left( LB\otimes \psi _{2}\left( B,B\right) \right) \circ \psi _{2}\left(
B,B\otimes B\right) $ is an epimorphism. By Lemma \ref{lem:bast}, we have $(\overline{A},q)\in\epindalg(\overline{B})$.
\end{proof}

\subsection{Constructing the Initial object in \texorpdfstring{$\epindobj(\overline{B})$}{TEXT}}

By Theorem \ref{teo:red}, under the relevant assumptions, to have an initial object in $\epindobj(\overline{B})$ is equivalent to having an initial object in $\indalg(\overline{B})$. By Proposition \ref{pro:induced2}, this is equivalent to having a colimit for \eqref{eq:14bis}, yielding then an explicit description  $\overline{L}\,
\overline{B}$. For this reason it is worthwhile to provide a construction of an initial object in $\epindobj(\overline{B})$. To this aim we first need to prove the following result.

\begin{lemma}
\label{lem:prod}
Let $I$ be a set and let $\left( \overline{E_i},q_{i}\right)_{i\in I}$ be a family of objects in  $\indobj(\overline{B})$. Assume that the family $\left( E_{i}\right) _{i\in I}$
has a product $\left( E,\left( p_{t}\right) _{t\in
I}\right) $ in $\Aa  $ and let $q:LB\rightarrow E$ be the unique morphism
such that $q_{i}=p_{i}\circ q$ for every $i.$ Then $E$ induces a tern $\overline{E}=(E,m_E,u_E)$ such that $(\left( \overline{E},q\right),(p_t)_{t\in I})$ is the product of the family $\left( \overline{E_{i}},q_{i}\right)_{i\in I}$ in $\indobj(\overline{B})$.

\end{lemma}

\begin{proof}By the universal property of the product, there are unique morphisms $%
q:LB\rightarrow E$, $m_E:E\otimes E\to E$ and $u_E:\unit\to E$ such that $p_{i}\circ q=q_{i}$, $p_i\circ m_E=m_{E_i}\circ (p_i\otimes p_i)$ and $p_i\circ u_E=u_{E_i}$, for every $i\in I$. Set $\overline{E}=(E,m_E,u_E)$. We have
\begin{align*}
p_{i}\circ q\circ Lm_{B} &=q_{i}\circ Lm_{B}=m_{E_{i}}\circ \left(
q_{i}\otimes q_{i}\right) \circ \psi _{2}\left( B,B\right) =m_{E_{i}}\circ \left( p_{i}\otimes p_{i}\right) \circ \left( q\otimes
q\right) \circ \psi _{2}\left( B,B\right) \\
&=p_{i}\circ m_{E}\circ \left( q\otimes q\right) \circ \psi _{2}\left(
B,B\right) , \\
p_{i}\circ q\circ Lu_{B} &=q_{i}\circ Lu_{B}=u_{E_{i}}\circ \psi
_{0}=p_{i}\circ u_{E}\circ \psi _{0}.
\end{align*}%
By the uniqueness in the universal property of the product, we get that
$$
q\circ Lm_{B} =m_{E}\circ \left( q\otimes q\right) \circ \psi _{2}\left(
B,B\right) , \qquad q\circ Lu_{B} =u_{E}\circ \psi _{0}.
$$
This proves that $\left( \overline{E},q\right) $ belongs to  $\indobj(\overline{B})$. Let us check it defines the desired product. From the equality $q_{t}=p_{t}\circ q$, we get that $p_t:E\to  E_t$ yields the projection $p_t:\left( \overline{E},q\right)\to \left( \overline{E_t},q_t\right)$ in $\indobj(\overline{B})$. Given, for every $t\in I$, a morphism $h_t:(\overline{A},p)\to(\overline{E_t},q_t)$ in $\indobj(\overline{B})$, by the universal property of $\left(
E,\left( p_{t}\right) _{t\in I}\right) $, there is a unique morphism $h:A\to E$ such that $p_{t} \circ h=h_t$. Since $p_t\circ h\circ p= h_t\circ p=q_t$, the uniqueness implies $h\circ p=q$ so that we get a morphism $h:(\overline{A},p)\to(\overline{E},q)$ that composed by the projection yields $h_t:(\overline{A},p)\to(\overline{E_t},q_t)$. Its uniqueness follows from the universal property of $\left(
E,\left( p_{t}\right) _{t\in I}\right) $.
\end{proof}


\begin{proposition}\label{pro:constinepi}
In the Setting \ref{not:setting}, assume that

\begin{itemize}
\item if $\left( \overline{A},q\right) \in\indobj(\overline{B})$, then $q$ admits an (Epi,
StrongMono)-factorization $q=h\circ q^{\prime }$;

\item the morphisms $\left( q^{\prime }\otimes q^{\prime }\right) \circ \psi
_{2}\left( B,B\right) $ and $\psi _{0}$ are epimorphisms;

\item there is set  $\Ss_{B}$ of objects of $\epindobj(\overline{B})$ such that each object   in  $\epindobj(\overline{B})$ is isomorphic to an element  in $\Ss_{B}$;

\item in $\Aa  $ there exists the product $E$ of the family $(D)_{(\overline{D},\delta_D)\in \Ss_{B}}$.

\end{itemize}
Then there is $\left( \overline{C},\delta _{C}\right)\in\mathcal{S}_{B}$ which is an initial object in $\epindobj(\overline{B})$ such that $\left( C,\delta _{C}\right)=\mathrm{Coim}(\delta)$ where $\delta :LB\rightarrow E$ is the diagonal
morphism of the family $\left( \delta _{D}\right) _{(\overline{D},\delta_D)\in \mathcal{S}_{B}}$.
\end{proposition}

\begin{proof}
By Lemma \ref{lem:posetalobj} and Proposition \ref{prop:initialepiobj}, $\epindobj(\overline{B})$ is a replete posetal weakly coreflective subcategory of $\indobj(\overline{B})$. Since the elements in $\Ss_{B}$ are, in particular objects in $\indobj(\overline{B})$, by Lemma \ref{lem:prod}, the object $E:=\prod\limits_{(\overline{D},\delta_D)\in \Ss_{B}}D$ induces a tern $\overline{E}=(E,m_E,u_E)$ such that $(\left( \overline{E},\delta\right),(p_D)_{(\overline{D},\delta_D)\in \Ss_{B}})$ is the product of the objects of $\Ss_{B}$ in $\indobj(\overline{B})$. By Corollary \ref{coro:constrinitial} applied to the set $\mathcal{S}_{B}$, we get that $( \overline{E},\delta)^\star$ is an initial object in $\epindobj(\overline{B})$. By Proposition \ref{prop:initialepiobj}, we know    that $( \overline{E},\delta)^\star=(\overline{E'},\delta')$ where $\left(E',\delta'\right)=\mathrm{Coim}(\delta)$. Since $(\overline{E'},\delta')=( \overline{E},\delta)^\star\in \epindobj(\overline{B})$, there is $(\overline{C},\delta_C)\in \Ss_{B}$ such that $(\overline{E'},\delta')\cong (\overline{C},\delta_C)$ as objects $\epindobj(\overline{B})$. Thus also $(\overline{C},\delta_C)$ is an initial object in $\epindobj(\overline{B})$ and $\left(C,\delta_C\right)=\mathrm{Coim}(\delta)$.
\end{proof}

\begin{proposition}
\label{pro:brut}In the Setting \ref{not:setting}, assume that

\begin{itemize}
\item the tensor products in $\Aa  $ preserve epimorphisms;

\item $\psi _{0}$ and the components of $\psi _{2}$ are epimorphisms in $\Aa  $;

\item if $( \overline{E},q)\in\indobj(\overline{B}) $, then $q$ admits an
(Epi,StrongMono)-factorization;

\item there is set  $\Ss_{B}$ of objects of $\epindobj(\overline{B})$ such that each object   in  $\epindobj(\overline{B})$ is isomorphic to an element  in $\Ss_{B}$;


\item in $\Aa  $ there exists the product $E$ of the family $(D)_{(\overline{D},\delta_D)\in \Ss_{B}}$.
\end{itemize}
Then there is $\left( \overline{C},\delta _{C}\right)\in\mathcal{S}_{B}$ which is an initial object in $\epindobj(\overline{B})$ such that $\left( C,\delta _{C}\right)=\mathrm{Coim}(\delta)$ where $\delta :LB\rightarrow E$ is the diagonal
morphism of the family $\left( \delta _{D}\right) _{(\overline{D},\delta_D)\in \mathcal{S}_{B}}$.

Finally, if the above assumptions hold for every algebra $\overline{B}$ in $\Bb$, then  $\overline{R}$ has a left adjoint $\overline{L}$ explicitly given by  $\overline{L}\,\overline{B}=\overline{ C}$.

\end{proposition}

\begin{proof}
Proposition \ref{pro:constinepi} ensures that there is $\left( \overline{C},\delta _{C}\right)\in\mathcal{S}_{B}$, as in the statement, which is an initial object in $\epindobj(\overline{B})$. By Theorem \ref{teo:red}, this is also an initial object in $\indalg(\overline{B})$. By Theorem \ref{teo:redinv}, we conclude.
\end{proof}

\begin{notations}\label{not:dualinduced}
For our purposes it is convenient to write Proposition \ref{pro:brut} in case $\Aa  =\Cc^\op $ for a covariant functor $L:\Bb  \to\Cc^\op $ regarded as a contravariant functor $(-)^\diamw:\Bb  \to\Cc$ such that $LB=(B^\diamw)^\op $ and $Lf=(f^\diamw)^\op $, for a morphism $f$. To this aim let us rewrite in $\Cc $ the notion of induced object in $\Aa
=\Cc ^{\op }$ of an algebra $\overline{B} =\left(
B,m_{B},u_{B}\right) $ in $\Bb  $. It consists of a pair $%
\left( \underline{E} ,e:E\rightarrow B^\diamw\right) $%
, where $\underline{E}=\left( E,\Delta _{E},\varepsilon _{E}\right)$ with $E$ an object in $\Cc $ and $\Delta _{E}:E\to E\otimes E,\varepsilon _{E}:E\to \unit $
and $e$ morphisms in $\Cc $ such that%
\begin{eqnarray}
(m_{B})^\diamw\circ e &=&\varphi _{2}\left( B,B\right) \circ \left( e\otimes
e\right) \circ \Delta _{E},\label{eq:goodDelta} \\
(u_{B})^\diamw\circ e &=&\varphi _{0}\circ \varepsilon _{E}\label{eq:goodeps}.
\end{eqnarray}%
where $\varphi _{2}\left( B,B\right):B^\diamw\otimes B^\diamw\to (B\otimes B)^\diamw$ and $\varphi _{0}:\unit \to \unit^\diamw $ are determined by $\varphi _{2}\left( B,B\right)^{\op }=\psi _{2}\left( B,B\right)$ and $\varphi _{0}^{\op }=\psi _{0}$ respectively. In this case we will say that $\left( \underline{E},e:E\rightarrow B^\diamw\right) $ is a \emph{good object} of $B^\diamw$ in $\Cc.$

Note that the induced object of $\overline{B}$ in $\Cc^\op $ corresponding to  $%
\left( \underline{E} ,e:E\rightarrow B^\diamw\right) $ is an epi-induced object if and only if $e^\op $ is an epimorphism in $\Cc^\op $ that is $e$ is a monomorphism in $\Cc$. In this case we will say that $\left( \underline{E},e:E\rightarrow B^\diamw\right) $ is a \emph{good subobject} of $B^\diamw$ in $\Cc.$ A morphism of good (sub)objects $h:\left( \underline{E},e:E\rightarrow B^\diamw\right)\to  \left( \underline{E'},e':E'\rightarrow B^\diamw\right)$ is a morphism $h:E\to E'$ such that $e'\circ h=e$. This way we get the category of good (sub)objects of $B^\diamw$ in $\Cc$ which turns out to be anti-isomorphic to the category of (epi-)induced objects of $\overline{B}$ in $\Cc^\op $.
\begin{invisible} This defines the category ${\sf GoodObj}(B^\diamw)$ of good objects.
 Let us  construct a category anti-isomorphism $${\sf GoodObj}(B^\diamw)\to\indobj(\overline{B}),\quad ((E,\Delta,\varepsilon),e))\mapsto ((E^\op,\Delta^\op,\varepsilon^\op),e^\op)),\quad h\mapsto h^\op.$$

Given $h:\left( \underline{E},e:E\rightarrow B^\diamw\right)\to  \left( \underline{E'},e':E'\rightarrow B^\diamw\right)$, we have that $e'\circ h=e$ so that $h^\op\circ (e')^\op=e^\op$ and hence we get a morphism $h^\op:((E'^\op,\Delta'^\op,\varepsilon'^\op),e'^\op))\to ((E^\op,\Delta^\op,\varepsilon^\op),e^\op))$. As a consequence ${\sf GoodObj}(B^\diamw)\cong\indobj(\overline{B})^\op$ and hence an initial object in $\indobj(\overline{B})$ yields a terminal object in ${\sf GoodObj}(B^\diamw)$.
\end{invisible}

\end{notations}

\begin{proposition}
\label{pro:brutop}
Let $%
\Cc $ and $\Bb  $ be monoidal categories and let $R:{\Cc^\op %
}\rightarrow \Bb  $ be a lax monoidal functor with a left adjoint $L,$
unit $\eta $ and counit $\epsilon $. In the Setting \ref{not:dualinduced}, assume that the functor $\mho :%
\Coalg \left( \Cc \right) \rightarrow \Cc $ has a right adjoint and that

\begin{itemize}
\item the tensor products in $\Cc $ preserve monomorphisms;

\item $\psi _{0}$ and the components of $\psi _{2}$ are monomorphisms in $\Cc $;

\item if $\left( \underline{E},e:E\rightarrow B^\diamw\right) $ is a good
object of $B^\diamw$ in $\Cc$, then the morphism $e$ admits an
(StrongEpi,Mono)-factorization in $\Cc$;

\item there is a set $\mathcal{S}_{B}$ of good subobjects of $B^\diamw$ in $\Cc$ such that each good subobjects of $B^\diamw$ in $\Cc$ is isomorphic to an element  in $\Ss_{B}$;

\item in $\Cc$ there exists the coproduct of the family $(D)_{(\underline{D},e_D)\in \Ss_{B}}$.
\end{itemize}

Then there is $(\underline{B^\diamb}, \theta_B:B^\diamb\to B^\diamw)\in\mathcal{S}_{B}$ which is a terminal good subobject of $B^\diamw$ in $\Cc$  such that $(B^\diamb, \theta_B)$ is the sum of the family of subobjects $(D,e_D)_{(\underline{D},e_D)\in \Ss_{B}}$ of $B^\diamw$.

Finally, if the above assumptions hold for every algebra $\overline{B}$ in $\Bb$, then $\overline{R}$ has a left adjoint $\overline{L}$ explicitly given by $\overline{L}\,\overline{B}=(\underline{B^\diamb})^\op$.
\end{proposition}

In the next section, we put all of the above developed theory to work to
explicitly compute lifted auto-adjunctions on categories of so-called
``color bialgebras''.


\section{Application: the group-graded case}\label{applicationsection}

This section is devoted to investigate the case of group-graded vector spaces. To this aim we first need to recall some auxiliary results connected to the notion of pre-rigid category.

\subsection{Pre-rigid monoidal categories}
 In order to discuss the examples of liftable functors of our concern, we  recall the following notion appeared in its original form in \cite[4.1.3]{GV-OnTheDuality}.

\begin{definition}\label{def:pre-rigid}
Following \cite[Definition 2.1]{AGM}, a monoidal category $(\Cc,\otimes,\unit)$ is
called \textit{pre-rigid} if for every object $X$ there exists an object $%
X^{\ast }$ and a morphism $\mathrm{ev}_{X}:X^{\ast }\otimes X\rightarrow
\unit $ (the \emph{evaluation at $X$}) with the following universal property: For every morphism $t:T\otimes X\rightarrow \unit$ there is a
unique morphism $t^\dag:T\rightarrow X^{\ast }$ such that $t=\ev%
_{X}\circ \left(t^\dag\otimes X\right) .$ Equivalently the map
\begin{equation*}
\hom _{\Cc }\left( T,X^{\ast }\right) \rightarrow \hom %
_{\Cc }\left( T\otimes X,\unit \right),\qquad u\mapsto \mathrm{ev}%
_{X}\circ \left( u\otimes X\right)
\end{equation*}%
is bijective for every object $T$ in $\Cc $.
\end{definition}

One has that a (right) closed monoidal category is pre-rigid (cf. \cite[%
Proposition 2.5]{AGM}). Notice that the converse is not true: the category
of bialgebras over a field $k$ for instance is pre-rigid monoidal \cite[%
Examples 2.19.3]{AGM}, but not closed. \newline

The following corollary will be applied to $\Cc =\Vec _{G}$.

\begin{proposition}
\label{pro:Men0} Let $\Cc $ be a pre-rigid braided monoidal
category. Assume that the forgetful functor $\mho:\Coalg (\Cc %
)\to\Cc $ has a right adjoint. Assume also that $\Coalg (%
\Cc )$ has equalizers.

Then $\left( (-)^{*}:\Cc \rightarrow \Cc ^{\op %
},(-)^{*}:\Cc ^{\op }\rightarrow \Cc \right)$ is a
liftable pair of adjoint functors.
\end{proposition}

\begin{proof}
Let $R=(-)^{*}:\Aa  \rightarrow \Bb  $ where $\Aa  :=\Cc ^{%
\op}$ and $\Bb  :=\Cc $. By the assumptions on $\Cc $, the category $\Aa  =\Cc ^{%
\op }$ fulfills the requirements of Theorem \ref{teo:Tambara}
and hence $\overline{R}=\Alg(R):{\Alg }({\Aa  }%
)\rightarrow {\Alg }({\Bb  })$ has a left adjoint $\overline{L}$. We conclude by \cite[Corollary 4.7]{AGM}. 
\end{proof}

If a pre-rigid monoidal category is also braided, we can construct on it an adjunction that under relevant assumptions results in a liftable pair of functors. More precisely we have the following results, which we record here for further use.

\begin{proposition} (cf. \cite[Proposition 4.4]{AGM})
\label{prop:extensionGV} When $\Cc $ is a pre-rigid braided monoidal
category, the assignment $X\mapsto X^*$ induces a functor $R=(-)^{*}:%
\Cc ^{\op }\rightarrow \Cc $ with a left adjoint $L=R^{%
\op }=(-)^{*}:\Cc \rightarrow \Cc ^{\op }.$
Moreover there are $\phi _{2},\phi _{0}$ such that $\left( R,\phi _{2},\phi
_{0}\right) $ is lax monoidal and, the induced colax monoidal structure on $L$ by  \eqref{form:PsiFromPhi2} and \eqref{form:PsiFromPhi0} is specifically $(\phi _{2}^{\op},\phi _{0}^{\op})$.
Explicitly, $\phi _{0}:\unit\rightarrow %
\unit^{\ast }$ is uniquely defined by  $\ev_{\unit}\circ \left( \phi _{0}\otimes \unit\right)
=m_{\unit}$ and  $\phi _{2}\left( X^{\op},Y^{\op}\right) :=\varphi _{2}\left( X,Y\right):X^{\ast }\otimes Y^{\ast
}\rightarrow \left( X\otimes Y\right) ^{\ast }$ by
\begin{equation}\label{form:varphi}
 \ev_{X\otimes Y}\circ (\varphi _{2}\left( X,Y\right)\otimes X \otimes Y)=
 \left( \ev_{X}\otimes \ev_{Y}\right) \circ ( X^{\ast }\otimes \left(c_{X,Y^{\ast }}\right) ^{-1}\otimes Y) .
\end{equation}
Moreover, for every $X$ in $\Cc $, the unit $\eta_X$ and the counit ${\epsilon }_{X^{\op }}=\left(j_{X}\right) ^{\op }$ of the adjunction $(L,R)$ are uniquely defined by the equalities
\begin{eqnarray}
\mathrm{ev}_{X}\circ c_{X,X^{\ast }} &=&\mathrm{ev}_{X^{\ast }}\circ \left(
\eta_{X}\otimes X^{\ast }\right) , \label{form:etaX}\\
\mathrm{ev}_{X}\circ \left( c_{X^{\ast },X}\right) ^{-1} &=&\mathrm{ev}_{X^{\ast }}\circ \left( j_{X}\otimes X^{\ast }\right)\label{form:jX} .
\end{eqnarray}
\end{proposition}

\begin{remark}\label{rem:Rbraided} We have noticed in Subsection \ref{sub:liftability} that a liftable pair induces an adjunction at the level of bialgebras in case the right adjoint is also braided. In case of Proposition \ref{prop:extensionGV},  the lax monoidal functor $\left( R,\phi _{2},\phi_{0}\right) $ is braided if and only if the following diagram commutes
  $$\xymatrixcolsep{3cm}
  \xymatrix{R(X^\op)\otimes R(Y^\op)\ar[r]^{c_{R(X^\op), R(Y^\op)}}\ar[d]_{\phi _{2}(X^\op, Y^\op)}& {R(Y^\op)\otimes R(X^\op)}\ar[d]^{\phi _{2}(Y^\op, X^\op)}\\
  R(X^\op\otimes Y^\op)\ar[r]^{Rc_{X^\op, Y^\op}}& {R(Y^\op \otimes X^\op)} }$$
  that is if and only if the following diagram commutes
  $$\xymatrixcolsep{3cm}
  \xymatrix{X^*\otimes Y^*\ar[r]^{c_{X^*, Y^*}}\ar[d]_{\varphi _{2}(X, Y)}& {Y^*\otimes X^*}\ar[d]^{\varphi _{2}(Y, X)}\\
  (X\otimes Y)^*\ar[r]^{(c_{X, Y}^{-1})^*}& (Y \otimes X)^* }$$
Set $f:=(c_{X, Y}^{-1})^*\circ \varphi _{2}(X, Y)$. We compute
\begin{align*}
  \ev_{Y\otimes X}\circ (f\otimes Y\otimes X)
  &=\ev_{Y\otimes X}\circ ((c_{X, Y}^{-1})^*\otimes Y\otimes X)\circ (\varphi _{2}(X, Y)\otimes Y\otimes X)\\
  &=\ev_{X\otimes Y}\circ ((X\otimes Y)^*\otimes c_{X, Y}^{-1})\circ (\varphi _{2}(X, Y)\otimes Y\otimes X)\\
  &=\ev_{X\otimes Y}\circ (\varphi _{2}(X, Y)\otimes X\otimes Y)\circ (X^*\otimes Y^*\otimes c_{X, Y}^{-1})\\
   &\overset{\eqref{form:varphi}}=\left( \ev_{X}\otimes \ev_{Y}\right) \circ ( X^{\ast }\otimes c_{X,Y^{\ast }} ^{-1}\otimes Y)\circ (X^*\otimes Y^*\otimes c_{X, Y}^{-1})\\
  &=\ev_X\circ (X^*\otimes \ev_Y\otimes X)\\
  &=(\ev_Y\otimes \ev_X)\circ (Y^*\otimes c_{Y,X^*}^{-1}\otimes X)\circ (c_{Y^*, X^*}^{-1}\otimes Y\otimes X)\\
 &\overset{\eqref{form:varphi}}=\ev_{Y\otimes X}\circ (\varphi _{2}(Y, X)\otimes Y\otimes X)\circ  (c_{Y^*, X^*}^{-1}\otimes Y\otimes X)
\end{align*}
so that $f=\varphi _{2}(Y, X)\circ c_{Y^*, X^*}^{-1}$ i.e. $(c_{X, Y}^{-1})^*\circ \varphi _{2}(X, Y)=\varphi _{2}(Y, X)\circ c_{Y^*, X^*}^{-1}$.

As a consequence, $\left( R,\phi _{2},\phi_{0}\right) $ is a braided monoidal functor if and only if $\varphi _{2}(Y, X)\circ c_{X^*, Y^*}=\varphi _{2}(Y, X)\circ c_{Y^*, X^*}^{-1}$ for all objects $X,Y$ in $\Cc$. Equivalently one has to ask that $\varphi _{2}(Y, X)\circ c_{X^*, Y^*}\circ c_{Y^*, X^*}=\varphi _{2}(Y, X)$ for all objects $X,Y$ in $\Cc$. In particular, if $\varphi _{2}$ is a monomorphism on components, this is equivalent to ask that $c_{X^*, Y^*}\circ c_{Y^*, X^*}=1_{Y^*\ot X^*}$ which is quite close to requiring that $\Cc$ is symmetric.
\end{remark}

Proposition \ref{prop:extensionGV} suggests a suitable context to obtain examples of liftable pairs of functors, as the following result shows.

\begin{proposition}\cite[Proposition 4.6]{AGM}
\label{lem:Barop}For a monoidal category $\Cc$, suppose a lax monoidal functor $\left( R,\phi _{2},\phi _{0}\right) :\Cc^{\op}\to \Cc$ has a left adjoint $L=R^{\op}$. If the induced colax monoidal structure on $L$ by  \eqref{form:PsiFromPhi2} and \eqref{form:PsiFromPhi0} is specifically $(\phi _{2}^{\op},\phi _{0}^{\op})$, then $\overline{R}=\left( \underline{L}\right) ^{\op}$.  Moreover, if $\overline{R}$ has a left adjoint,
then $\left( L,R\right) $ is liftable.
\end{proposition}


\begin{remark}
\label{rem:regop}Keeping the hypotheses of Proposition \ref{lem:Barop}  and
assuming that $\Aa  =\Cc ^{\op }$, the functor $\mho :%
\Coalg \left( \Cc \right) \rightarrow \Cc $ has a right
adjoint and $\Coalg \left( \Cc \right)$ has equalizers, then, by Theorem \ref{teo:Tambara}, the functor $\overline{R}$ has a left adjoint $\overline{L}$ which we will now describe.
We have
\begin{equation*}
\xymatrix{ \Aa= \Cc ^\op  \ar@<.5ex>[rr]^-{R} && \Bb=\Cc
\ar@<.5ex>[ll]^-{L}. }
\end{equation*}
By assumption, $\left( R,\phi _{2},\phi _{0}\right) $ is lax monoidal and it has a left adjoint $L=R^{\op}$. Moreover, the induced colax monoidal structure $\left(\psi _{2},\psi _{0}\right)$ on $L$ by  \eqref{form:PsiFromPhi2} and \eqref{form:PsiFromPhi0} is required to be specifically $(\phi _{2}^{\op},\phi _{0}^{\op})$. As in the Setting \ref{not:dualinduced} we can regard $L$ as a contravariant functor $(-)^\diamw:\Cc\to \Cc$ and define $\varphi _{2}\left( B,B\right):B^\diamw\otimes B^\diamw\to (B\otimes B)^\diamw$ and $\varphi _{0}:\unit \to \unit^\diamw $ by setting $\varphi _{2}\left( B,B\right)^{\op }=\psi _{2}\left( B,B\right)$ and $\varphi _{0}^{\op }=\psi _{0}$ respectively.
Note that, in view of the requirement  $\left(\psi _{2},\psi _{0}\right)=\left( \phi _{2}^{\op },\phi _{0}^{\op }\right)$, we also have $\varphi _{2}\left( B,B\right):=\phi _{2}\left( B^{\op },B^{\op }\right)$ and $\varphi _{0}=\phi _{0}$.

Assume further that
\begin{itemize}
\item the tensor products preserve monomorphisms in $\Cc $;

\item $\phi _{0}:\unit \rightarrow \unit^\diamw $ is invertible and
the components of $\varphi _{2}\ $are monomorphisms;

\item for every $X,Y\in\Cc $, any morphism $X\rightarrow Y^\diamw$ in
$\Cc $ has a (StrongEpi,Mono)-factorization;

\item for every algebra $\overline{B}$ in $\Bb$, there is a set $\mathcal{S}_{B}$ of good subobjects of $B^\diamw$ in $\Cc$ such that each good subobjects of $B^\diamw$ in $\Cc$ is isomorphic to an element  in $\Ss_{B}$.
\end{itemize}

Since $\phi _{0}$ is invertible, \eqref{eq:goodeps}  rewrites as $\varepsilon _{E}=\phi _{0}^{-1}\circ
(u_{B})^\diamw\circ e$, so that $\varepsilon _{E}$ is completely determined and can
be ignored in the definition of good object. Thus it suffices to consider terns $\left( E,\Delta
_{E},e\right) $ such that \eqref{eq:goodDelta} is fulfilled.

All the above assumptions permit to apply Proposition \ref{pro:brutop}.

Then there is $(\underline{B^\diamb}, \theta_B:B^\diamb\to B^\diamw)\in\mathcal{S}_{B}$ which is a terminal good subobject of $B^\diamw$ in $\Cc$  such that $(B^\diamb, \theta_B)$ is the sum of the family of subobjects $(D,e_D)_{(\underline{D},e_D)\in \Ss_{B}}$ of $B^\diamw$.

Finally $\overline{R}$ has a left adjoint $\overline{L}$ explicitly given by $\overline{L}\,\overline{B}=(\underline{B^\diamb})^\op$.
\end{remark}

The following result is a pre-rigid version of
\cite[Proposition 8]{PS}. Note that, given a closed monoidal category $(\Cc,\otimes,\unit)$, the right adjoint $[-,\unit]_r$ of the functor $(-)\otimes \unit$ induces the functor $R=[-,\unit]_r:\Cc^\op\to \Cc$. The authors therein call $\overline{R}=\Alg(R)$ the \emph{dual monoidal functor} and prove it has a left adjoint in case $\Cc$ is locally presentable. Note also that a closed monoidal category is in particular pre-rigid with pre-dual given by $(-)^*:=[-,\unit]_r$.

\begin{corollary}
\label{coro:Men2} Let $\Cc $ be a pre-rigid braided monoidal
category. Assume that $\Cc $ is locally presentable and that the
tensor products preserve directed colimits.\newline
Then $\left( (-)^{*}:\Cc \rightarrow \Cc ^{\op %
},(-)^{*}:\Cc ^{\op }\rightarrow \Cc \right)$ is a
liftable pair of adjoint functors.
\end{corollary}

\begin{proof}
Since $\Cc $ is monoidal and locally presentable and since the tensor
products preserve directed colimits, by the proof of \cite[page 8]{Po}
(which does not use the symmetry assumption present in the definition of
admissible category), we have that $\Coalg (\Cc )$ is locally
presentable and comonadic over $\Cc $. In particular the functor $%
\mho:\Coalg (\Cc )\to\Cc $ has a right adjoint. By \cite[Corollary 1.28]{AR}, the category $\Coalg (\Cc )$ is complete
so that it has equalizers. We conclude by Proposition \ref{pro:Men0}.
\end{proof}

\begin{remark}
In the setting of  Corollary \ref{coro:Men2}, since $\Cc $ is a locally presentable category, 
it has (StrongEpi, Mono)-factorization of morphisms.
Thus, part of the conditions of Remark \ref{rem:regop} are automatically satisfied.
\end{remark}

${\Vec }_{G}$ is an example of a locally presentable category by \cite%
[Theorem 10]{Po0}. Since the tensor product in $\Vec _{G}$ is $%
\otimes _{k},$ it preserves directed colimits. Thus one can also apply
Corollary \ref{coro:Men2} to this category. Indeed locally presentability is even too much to gain the liftability of the adjunction induced by the pre-dual because, as we will see, the hypotheses of Proposition \ref{pro:Men0} are sufficient.

\subsection{A group-graded version of Sweedler's finite dual}\label{GVec}

We now take a closer look at two examples, by first taking $\Cc =%
\Vec $, then by taking $\Cc =%
\Vec _{G}$ which is our case of main
interest.\medskip

\textbf{The vector space case}. Let us consider the case of vector spaces, putting $\Cc =\Vec $, which is a pre-rigid braided monoidal category, the pre-dual of a vector space $V$ being given by the linear dual $V^{\ast }:=\hom _{k}\left( V,k\right)$. By Proposition \ref{prop:extensionGV}, we have a functor $R=(-)^*:\Cc^\op \rightarrow \Cc ,X\mapsto
X^{\ast }$ with left adjoint $L=R^\op$. The maps $\varphi _{2}$ and $\phi _{0}$ of Proposition \ref{prop:extensionGV} are defined by
\begin{align*}
\varphi _{2}\left( X,Y\right) &:X^{\ast }\otimes Y^{\ast }\rightarrow
\left( X\otimes Y\right) ^{\ast },\quad f\otimes g\mapsto m_{\Bbbk }\left(
f\otimes g\right) , \\
\phi _{0} &:k\rightarrow k^{\ast },\quad a\mapsto a{1_{k}}.
\end{align*}
Note that all the requirements of Remark \ref{rem:regop} are satisfied (in particular all epimorphisms are regular whence strong), where, given
an algebra $\overline{B}=\left( B,m_{B},u_{B}\right) $ in $\Cc ,$ we let $\mathcal{S}_{B}$ be the set of all good
subspaces of $B^{\ast }$ (we will use the word subspace when the monomorphism
is an inclusion). As a consequence,
$\overline{L}\,\overline{B}=(\underline{B^\circ})^\op$ where $B^\circ $ is the sum of all good
subspaces of $B^{\ast }$.
By \cite[pages 19-20]{Michaelis-LieCoalg} we know that $B^\circ $ is
exactly the Sweedler's finite dual of $B$. \medskip

\textbf{The group-graded case}. Let $G$ be an abelian group, with neutral element $e$ and
let $\Vec _{G}$ be the category whose objects are vector spaces (over
a field $k$) graded by the group $G$. For objects $V=\oplus_{g\in G}V_g,
W=\oplus_{g\in G}W_g \in \Vec _{G}$, the set of morphisms in $\mathsf{%
Vec}_G$ (i.e. degree-preserving $k$-linear maps) will be denoted as $\mathrm{%
Hom}(V,W)$. The category $\Vec _{G}$ admits a monoidal structure,
which we now briefly recall. If $V,W \in \Vec _{G}$, then $V\otimes
W:= \bigoplus_{g}(\oplus_{xy=g} V_x \otimes_k W_y)$ becomes an object in $%
\Vec _G$. The unit object is $k=k_e$. Taking associativity and unit
constraints to be trivial, $(\Vec _G, \otimes, k)$ indeed becomes a
monoidal category. \newline
Note that the monoidal category $\Vec_G$ is (right) closed. In fact we can consider the right adjoint to the endofunctor $\left( -\right) \otimes V$ of $\Vec _{G}$ (tensor product of graded vector spaces) and denote this
adjoint by $\mathrm{HOM}\left( V,-\right) $.
Since $G$ is abelian, we have that $\mathrm{HOM}\left( V,W\right)
=\oplus _{g\in G}\mathrm{HOM}\left( V,W\right) _{g}$ where $$\mathrm{HOM}%
\left( V,W\right) _{g}=\left\{ f\in \hom _{k}\left( V,W\right) \mid
f\left( V_{h}\right) \subseteq W_{hg}\text{ for every }h\in G\right\},$$ for
any $V,W\in \Vec _{G}$.
As we already mentioned that a (right) closed monoidal category is pre-rigid, we get that $\Vec %
_G$ is pre-rigid. To avoid confusion with the
usual (non-graded) linear dual of a vector space, the pre-dual of a $G$-graded vector space $V=\oplus_g V_g$ in $\Vec_{G}$  will be denoted by $V^{\diamw}$. Thus $V^{\diamw}:=\mathrm{HOM}%
\left( V,k\right) .$ We can write explicitly the graduation of $V^{\diamw}$ as
\begin{eqnarray*}
\left( V^{\diamw}\right) _{g} &=&\mathrm{HOM}\left( V,k\right) _{g}=\left\{
f\in \hom _{k}\left( V,k\right) \mid f\left( V_{h}\right) \subseteq
k_{hg}\text{ for every }h\in G\right\} \\
&=&\left\{ f\in \hom _{k}\left( V,k\right) \mid f\left( V_{h}\right)
=0\text{ for every }h\in G,h\neq g^{-1}\right\} \cong \hom _{k}\left(
V_{g^{-1}},k\right)=(V_{g^{-1}})^* .
\end{eqnarray*}%
In order to discuss braided structures on $\Vec _G$, recall that a
bicharacter on $G$ is a map $\alpha: G\times G\to k\setminus\{0\}$ such that
\begin{equation*}
\alpha(gh,l)=\alpha(g,l)\alpha(h,l)\qquad\mathrm{and}\qquad\alpha(g,hl)=\alpha(g,h)%
\alpha(g,l),\qquad \text{for all } g,h,l \in G.
\end{equation*}
Letting $\alpha$ be a bicharacter, we can define a braiding $c^{\alpha}$ on $%
\Vec _G$, given on homogeneous objects by
\begin{equation*}
c^{\alpha}_{V_g,W_h}: V_g\otimes W_h\to W_h\otimes
V_g,\qquad v\otimes w\mapsto \alpha(g,h)w\otimes v.
\end{equation*}
We notice that, in order for $c^{\alpha}$ to be a morphism in $\Vec _G$, we need that $G$ is abelian. Remark also that $%
c^{\alpha}$ is a symmetry if and only if moreover holds that $%
\alpha(g,h)\alpha(h,g)=1, ~\forall g,h\in G$. In this case we say that $\alpha$ is \emph{skew-symmetric}. We shall denote the
thus-obtained braided monoidal category as $\Vec _G^{\alpha}$.
\newline
Since $\Vec _G^{\alpha}$ is a braided and pre-rigid monoidal category, we can use Proposition \ref{prop:extensionGV}
to get that
\begin{equation*}
R:=(-)^{\diamw}: (\Vec _{G}^{\alpha})^{\op }\to \Vec %
_{G}^{\alpha}
\end{equation*}
is a self-adjoint lax monoidal functor, for any bicharacter $\alpha$.
\newline
Now, using \cite[Corollary 4.6]{AI} (note that $\Vec_{G}$ can be regarded as the category of comodules over the
group-algebra $kG$), the forgetful functor $\mho: \Coalg(\Vec _{G})\to \Vec_G$
has a right adjoint. Moreover any parallel pair $f,g:C\to D$ in $\Coalg (\Vec _{G})$ has equalizer given by  $$\{c\in C \mid \sum c_1\otimes_k f(c_2)\otimes_k c_3=\sum c_1\otimes_k g(c_2)\otimes_k c_3\},$$ where we are using Sweedler's notation for the comultiplication\footnote{note that it coincides with the equalizer of the same pair in $\Coalg (\Vec)$, see e.g. \cite[Remark 1.2]{Agore}.}.

As a consequence, by Proposition \ref{pro:Men0} we can conclude that the adjoint pair
of functors $(L,R)$ introduced above is liftable.

Although it does not seem to appear in
literature, the left adjoint $\overline{L}$ of $\overline{R}=\Alg(R)$ -whose existence is part of the definition of a liftable pair of adjoint functors- can be described explicitly. It is our purpose here to do so. Indeed, it is shown below that, given an algebra $%
\overline{B}=\left( B,m_{B},u_{B}\right) $ in $\Vec _{G},$ the object $\overline{%
L}\,\overline{B}$ can be identified with the biggest ``good'' $G$-graded
subspace of $B^{\diamw}$. More precisely, $\overline{L}\,\overline{B}=\left(
\underline{B^\diamb }\right) ^{\op }$ where
\begin{equation*}
B^\diamb=\left\{ f\in B^{\diamw}\mid \xi _{B}\left( f\right) \text{
vanishes on some }I\in \mathcal{I}_{B}^{f}\right\},
\end{equation*}
$\xi_{B}:B^{\diamw}=\mathrm{HOM}\left( B,k\right)\to \hom _{k}(B,k)=B^*$
denoting the canonical injection and $\mathcal{I}_{B}^{f}$ being the set of
finite-codimensional $G$-graded ideals of $\overline{B}.$

\begin{remark}
\label{rem:vec} Note that, if we take $G$ to be the trivial group in the
above discussion, we recover the case of vector spaces. When taking $%
G=\langle g | g^{2}=e \rangle$, the cyclic group of order two, and $\alpha$
trivial everywhere except for $\alpha(g,g)=-1$, one obtains the super vector
space case; this incorporates \cite[Remark 3.1]{GV-OnTheDuality}.
\end{remark}

Let us proceed with the details of the computation of $\overline{L}\,\overline{B}$. 

Note that, given objects $T$ and $X$ in $\Vec_G$, for every morphism $t:T\otimes X\to k$, the map $t^\dag:T\to X^{\diamw }$ of Definition \ref{def:pre-rigid} is uniquely determined by the equality $t=\mathrm{ev}_{X}\circ (t^\dag\otimes X)$ i.e. $t^\dag(a)=t(a\otimes -)$ for $a\in T$.
Let us start by describing explicitly some of the maps given in Proposition \ref{prop:extensionGV}.

\begin{lemma}\label{lem:phiG}
The map $\varphi _{2}\left( X,Y\right) :X^{\diamw }\otimes
Y^{\diamw }\rightarrow \left( X\otimes Y\right) ^{\diamw }$ is given, for $f\in \left( X^{\diamw
}\right) _{a},g\in \left( Y^{\diamw }\right) _{b},$ by $\varphi _{2}\left( X,Y\right) \left( f\otimes g\right) :=\alpha \left(
a,b\right) m_{k}\left( f\otimes g\right) .$ The map $\phi _{0}:k\rightarrow k^{\diamw }$ is given, for $\lambda \in k,$ by the
equality $\phi _{0}\left( \lambda \right) =\lambda 1_{k}.$ In particular $\phi _{0}$ is invertible and the components of $\varphi _{2}$ are monomorphisms.
\end{lemma}

\begin{proof}
By
Proposition \ref{prop:extensionGV}, $\varphi _{2}\left( X,Y\right) $ is given, for $f\in \left( X^{\diamw
}\right) _{a},g\in \left( Y^{\diamw }\right) _{b},$ by the equality\begin{equation*}
\varphi _{2}\left( X,Y\right) \left( f\otimes g\right) = \left(
\mathrm{ev}_{X}\otimes \mathrm{ev}_{Y}\right) \left( X^{\diamw }\otimes \left(
c_{X,Y^{\diamw }}\right) ^{-1}\otimes Y\right) \left( f\otimes g\otimes
-\right) .
\end{equation*}Given $x\in X_{c},y\in Y_{d}$, we have $f(x)=\delta _{ca,e}f(x)$ and $g(y)=\delta _{db,e}g(y)$ so that
\begin{eqnarray*}
\varphi _{2}\left( X,Y\right) \left( f\otimes g\right) \left( x\otimes
y\right) &=&\left[ \left( \mathrm{ev}_{X}\otimes \mathrm{ev}_{Y}\right)
\circ \left( X^{\diamw }\otimes \left( c_{X,Y^{\ast }}\right) ^{-1}\otimes
Y\right) \right] \left( f\otimes g\otimes x\otimes y\right) \\
&=&
\alpha \left(
c,b\right) ^{-1}f\left( x\right) g\left( y\right) = \delta _{ca,e}\delta _{db,e}\alpha \left( c,b\right) ^{-1}f\left(
x\right) g\left( y\right) \\
&=&
\alpha \left( a,b\right) f\left( x\right) g\left( y\right)
=\alpha \left( a,b\right)
m_{k}\left( f\otimes g\right) \left( x\otimes y\right) .
\end{eqnarray*}Thus $\varphi _{2}\left( X,Y\right) \left( f\otimes g\right) :=\alpha \left(
a,b\right) m_{k}\left( f\otimes g\right) .$
On the other hand $\phi _{0}$ is uniquely determined, for $\lambda \in k,$ by the
equality $
\phi _{0}\left( \lambda \right) =m_{k}\left( \lambda \otimes -\right)
=\lambda 1_{k}.$
\end{proof}

\begin{remark}
We already noticed that the lax monoidal functor $R:=(-)^{\diamw}: (\Vec _{G}^{\alpha})^{\op }\to \Vec %
_{G}^{\alpha}$ induced by the pre-dual is part of a liftable adjoint pair
of functors $(L,R)$. Moreover, by Lemma \ref{lem:phiG} we know that the components of $\varphi _{2}$ are monomorphisms. Thus, in view of Remark \ref{rem:Rbraided}, we have that $R$ is braided if and only if $c^\alpha_{X^*, Y^*}\circ c^\alpha_{Y^*, X^*}=1_{Y^*\ot X^*}$ for all objects $X,Y$ in $\Cc$. In particular this holds if $c^\alpha$ is a symmetry, which happens if and only if $\alpha$ is skew-symmetric. In this case, we get an induced auto-adjunction $(%
\underline{\overline{L}},\underline{\overline{R}})$ on the category of
bialgebras in $\Vec _{G}^{\alpha}$, i.e. ``color bialgebras'' (in the
sense of \cite[Section 1.4]{AAB} e.g.), for any such a bicharacter $\alpha$.
\end{remark}

Note that all the requirements of Remark \ref{rem:regop} are satisfied, by
the discussion above where we chose $\mathcal{S}_{B}$ to be the set of all good $G$-graded subspaces of $%
B^{\diamw }$ (as in case of $\Vec$, we will use the word subspace when the monomorphism
is an inclusion).
Thus, given an algebra $\overline{B}=\left(
B,m_{B},u_{B}\right) $ in $\Cc =\Vec _{G},$ we get that $%
\overline{L}\,\overline{B}$ can be identified with the biggest good $G$-graded
subspace of $B^{\diamw }$. Explicitly, there is $(\underline{B^\diamb}, \theta_B:B^\diamb\to B^\diamw)\in\mathcal{S}_{B}$, where
$\theta _{B}:B^\diamb\rightarrow B^{\diamw }$ is the canonical inclusion, which is a terminal good $G$-graded subspace of $B^\diamw$  such that $(B^\diamb, \theta_B)$ is the sum of the family of $G$-graded subspaces $(D,e_D)_{(\underline{D},e_D)\in \Ss_{B}}$ of $B^\diamw$.
Finally $\overline{L}$ is given by $\overline{L}\,\overline{B}=(\underline{B^\diamb})^\op$.\medskip

To round off this example, let us further refine the description of $(B^\diamb, \theta_B)$.
To this aim denote by $\mathcal{S}_{B}^{f}$ the set of
finite-dimensional good $G$-graded subspaces of $B^{\diamw }.$

\begin{lemma}\label{lem:bullet1}
$(B^\diamb, \theta_B)$ is the sum of the family of finite-dimensional $G$-graded subspaces $(D,e_D)_{(\underline{D},e_D)\in \Ss_{B}^f}$ of $B^\diamw$.
\end{lemma}

\begin{proof}
 Let $\left( E,\Delta _{E},e\right) $ be a good $G$-graded subspace of $B^{\diamw },$ where $e:E\rightarrow B^{\diamw }$ denotes the canonical
inclusion. Then the right-hand side square in the following diagram commutes by \eqref{eq:goodDelta}.
\begin{equation*}
\begin{array}{ccccccc}
\xymatrix{C\ar[d]_ {\Delta _{C} }\ar[rr]^-{\gamma}  && {E}\ar[d]^{\Delta _{E}} \ar[rr]^-{e} && B^{\diamw} \ar[rr]^-{(m_{B})^{\diamw}} && \left( B\otimes
B\right) ^{\diamw }\ar[d]^{1_{\left( B\otimes
B\right) ^{\diamw }}}\\
C\otimes C  \ar[rr]^-{\gamma \ot \gamma} && E\otimes E \ar[rr]^-{e\otimes e} && B^{\diamw }\otimes B^{\diamw }  \ar[rr]^-{\varphi _{2}\left( B,B\right)} && \left( B\otimes
B\right) ^{\diamw }
}
\end{array}
\end{equation*}Consider a subcoalgebra $\left( C,\Delta _{E},\varepsilon_E,\gamma :C\rightarrow E\right)
. $ Hence the external diagram above commutes. This means that $\left(
C,\Delta _{E},e\circ \gamma \right) $ is a good $G$-graded subspace of $B^{\diamw }.$ This proves that a subcoalgebra of a good $G$-graded subspace of
$B^{\diamw }$ is a good subspace of $B^{\diamw }$.

Given $(\underline{D},e_D)\in \Ss_{B}$, we know that $\underline{D}$ becomes a $G$-graded coalgebra, cf. the dual result of Lemma \ref{lem:bast}. Hence we can apply \cite[Theorem 4.5]{AI} to get that each $\underline{D}$ is sum of finite-dimensional $G$-graded subcoalgebras.
Since a subcoalgebra of a good subspace of $B^{\diamw }$ is again a good
subspace of $B^{\diamw },$ we can write $(B^\diamb, \theta_B)$ as the desired sum. \end{proof}

In order to describe the elements in $\mathcal{S}_{B}^{f}$, we first need the following lemma, which further specifies other maps involved in Proposition \ref{prop:extensionGV}.

\begin{lemma}\label{lem:etajective}
The morphisms $\eta_{X},j_{X}:X\rightarrow X^{\diamw \diamw }$ are given, for $x\in X_{a}$, $f\in \left( X^{\diamw
}\right) _{b}$, by
\begin{eqnarray*}
\eta_{X}\left( x\right) \left( f\right) &=&\alpha \left( a,a\right)
^{-1}f\left( x\right) ,\text{ for }x\in X_{a},f\in X^{\diamw }, \\
j_{X}\left( x\right) \left( f\right) &=&\alpha \left( a,a\right) f\left(
x\right) ,\text{ for }x\in X_{a},f\in X^{\diamw }.
\end{eqnarray*}Moreover these $\eta_{X}$ and $j_{X}$ are both injective.
\end{lemma}

\begin{proof}
By Proposition \ref{prop:extensionGV}, for $x\in X_{a}$, $f\in \left( X^{\diamw
}\right) _{b}$, we have
\begin{align*}
\eta_{X}\left( x\right) \left( f\right) &\overset{\eqref{form:etaX}}{=}\mathrm{ev}_{X}c_{X,X^{\diamw }}\left(
x\otimes f\right) =\alpha \left( a,b\right) f\left( x\right) =\delta
_{ab,e}\alpha \left( a,b\right) f\left( x\right) =\alpha \left(
a,a\right)^{-1} f\left( x\right), \\
j_{X}\left( x\right) \left( f\right) &\overset{\eqref{form:jX}}{=}\mathrm{ev}_{X}\left( c_{X^{\diamw
},X}\right) ^{-1}\left( x\otimes f\right) =\alpha \left( b,a\right)
^{-1}f\left( x\right) =\delta _{ab,e}\alpha \left( b,a\right) ^{-1}f\left(
x\right) =\alpha \left( a,a\right) f\left( x\right).
\end{align*}Let us check the injectivity. Let $x\in X_{a}$ be nonzero. If we complete $x$ to a basis of $X$ we can
consider the map $\lambda _{x}\in \hom _{k}\left( X,k\right) $ such
that $\lambda _{x}\left( x\right) =1$ and $\lambda _{x}$ vanishes on the
other elements of the basis. By construction $\lambda _{x}\in \mathrm{HOM}\left( X,k\right) _{a^{-1}}.$ Thus we can compute $\eta_{X}\left( x\right)
\left( \lambda _{x}\right) =\alpha \left( a,a\right) ^{-1}\lambda _{x}\left(
x\right) =\alpha \left( a,a\right) ^{-1}\neq 0.$ This proves that $\left(
\eta_{X}\right) _{a}$ is injective and hence $\eta_{X}$ is injective. Similarly
one gets that $j_{X}$ is injective.
\end{proof}

We are now ready to provide the promised description of $B^\diamb$.

\begin{proposition}\label{pro:bullet2}
We have that $
B^\diamb=\left\{ f\in B^{\diamw }\mid \xi _{B}\left( f\right) \text{
vanishes on some }I\in \mathcal{I}_{B}^{f}\right\}$, where $\mathcal{I}_{B}^{f}$ denotes the set of
finite-codimensional $G$-graded ideals of $\overline{B}$ and $\xi_B:B^{\diamw }\to B^*$ is the canonical injection.
\end{proposition}

\begin{proof}
Let $\left(
E,\Delta _{E},e\right) $ be a finite-dimensional good $G$-graded subspace of
$B^{\diamw }.$ Since $E$ is finite-dimensional we have that $E^{\diamw }=\mathrm{HOM}\left( E,k\right) =\hom \left( E,k\right) =E^{\ast }$ (see \cite[Lemma
3.3.2]{NaVa}). Thus in this case $\eta_{E},j_{E}:E\rightarrow E^{\diamw \diamw }=E^{\ast \ast }.$ Since, by Lemma \ref{lem:etajective}, these maps
are injective maps between spaces with the same dimension, we deduce that $\eta_{E},j_{E}$ are invertible.

Since $(\underline{B^\diamb}, \theta_B:B^\diamb\to B^\diamw)\in\mathcal{S}_{B}$ is a terminal good $G$-graded subspace of $B^\diamw$, there is a $G$-graded coalgebra map $\underline{f}:\underline{E}\rightarrow \underline{B^\diamb}$ such that $\theta _{B}\circ f=e.$ Thus we get the algebra morphism $(\underline{f})^\op:(\underline{B^\diamb})^\op\rightarrow (\underline{E})^\op$. If we regard $(\underline{E}, e)$ as the epi-induced object $\left( \overline{E^\op }=\underline{E}^\op ,e^\op \right) $ of $\overline{B}$, we can rewrite this morphism as $\overline{f^\op }:\overline{L}\,\overline{B}\rightarrow \overline{E^\op }$ where $\overline{f^\op }=(\underline{f})^\op$. Then, if we recall that the canonical projection $p:LB\to  \Omega\overline{L}\,\overline{B}$ is just $p=\theta _{B}^\op$ and we apply to $\overline{f^\op }:\overline{L}\,\overline{B}\rightarrow \overline{E^\op }$ the following adjunction, where $\Aa^\op=\Bb=\Vec_G^\alpha$,
\begin{equation*}
\hom _{\Alg \left( \Aa  \right) }\left( \overline{L}\,\overline{B},\overline{E^\op}\right) \cong \hom _{\Alg \left( \Bb  \right) }\left(
\overline{B},\overline{R}\,\overline{E^\op}\right) :h\mapsto \overline{R}h\circ
\overline{\eta }_{\overline{B}},
\end{equation*} we get the algebra morphism $\overline{\tau_B}:=\overline{R}\,\overline{f^\op }\circ
\overline{\eta }_{\overline{B}}:\overline{B}\to\overline{R}\,\overline{E^\op}$. To get a better description of this morphism, recall that the unit $\overline{\eta }$ of $(\overline{L},\overline{R})$ is given by $\Omega \overline{\eta }_{\overline{B}}=R' \kappa  _{\overline{B}}\circ \eta' _{B}$ where $R':=R\Omega$ and $\eta':=R\alpha L\circ\eta$ (see the proofs of Proposition \ref{pro:TambaraGen} and Theorem \ref{teo:Tambara}). Moreover, by Proposition \ref{pro:induced2}, the morphism $p:LB\to \Omega\overline{L}\,\overline{B}$ can be written in terms of $\kappa_{\overline{B}}:TLB\to \overline{L}\,\overline{B}$ as $p=\Omega \kappa_{\overline{B}}\circ \alpha_{LB}$. Therefore we get
\begin{equation*}
\Omega \overline{\eta }_{\overline{B}}=R\Omega \kappa _{\overline{B}}\circ R\alpha _{LB}\circ \eta _{B}=Rp\circ \eta _{B}=\theta
_{B}^{\diamw }\circ \eta _{B}:B\rightarrow R\Omega \overline{L}\,\overline{B}=\Omega \overline{R}\overline{L}\,\overline{B}=B^{\diamb \diamw },
\end{equation*}
and hence\begin{equation*}
\tau _{B}=\Omega\overline{\tau_B}= R\left( f^\op \right) \circ \Omega \overline{\eta }_{\overline{B}}=f^{\diamw }\circ \theta _{B}^{\diamw }\circ \eta _{B}=\left( \theta _{B}\circ
f\right) ^{\diamw }\circ \eta _{B}=e^{\diamw }\circ \eta _{B}:B\rightarrow E^{\diamw }.
\end{equation*}Since $\overline{\tau_B}$ is a $G$-graded algebra map, then the kernel of the map $\tau _{B}$, say  $I $, is obviously a $G$-graded ideal
of $\overline{B}.$ Consider the following exact sequence\begin{equation*}
0\rightarrow I\xrightarrow{i_{I}}B\xrightarrow{p_{I}}
\frac{B}{I}\rightarrow 0.
\end{equation*}Since it is a sequence in $\Vec _{G}$ -which is a semisimple
category- applying the contravariant functor $\left( -\right) ^{\diamw },$
we get the exact sequence\begin{equation}\label{seq:ex}
0\rightarrow \left( \frac{B}{I}\right) ^{\diamw }\xrightarrow{p_{I}^{\diamw }}B^{\diamw }\xrightarrow{i_{I}^{\diamw }}I^{\diamw
}\rightarrow 0.
\end{equation}Since $I=\mathrm{Ker}\left( \tau _{B}\right) $, there is a $G$-graded algebra injection $\lambda _{I}:\frac{B}{I}\rightarrow E^{\diamw }$ such that
\begin{equation*}
\lambda _{I}\circ p_{I}=\tau _{B}=e^{\diamw }\circ \eta_{B}.
\end{equation*}Since $E$ is finite-dimensional, so is $E^{\diamw }$ and hence $B/I$ is, too. This
shows that $I$ has finite codimension whence $I\in \mathcal{I}_{B}^{f}$. Define the map\begin{equation*}
\chi _{B}:=\left( E\xrightarrow{j_{E}}E^{\diamw \diamw }\xrightarrow{\left( \lambda _{I}\right) ^{{\diamw }}}\left( \frac{B}{I}\right) ^{\diamw }\right) .
\end{equation*}Note that $\chi _{B}$ is surjective as $j_{E}$ is invertible and $\left(
\lambda _{I}\right) ^{{\diamw }}$ is surjective. We compute\begin{equation*}
\left( p_{I}\right) ^{{\diamw }}\circ \chi _{B}=\left( p_{I}\right) ^{{\diamw
}}\circ \left( \lambda _{I}\right) ^{{\diamw }}\circ j_{E}=\left(
\eta_{B}\right) ^{\diamw }\circ e^{\diamw \diamw }\circ j_{E}=\left( \eta_{B}\right)
^{\diamw }\circ j_{B^{\diamw }}\circ e=e
\end{equation*}so that the following diagram commutes\begin{equation*}
\xymatrix{E\ar[rr]^-{\chi _{B}}_-{\cong } \ar[rrd]_{e} &&\left( \frac{B}{I}\right) ^{\diamw } \ar[d]^{p_{I}^{\diamw }=\left( p_{I}\right) ^{{\diamw }} }\\
&& B^{\diamw}  }
\end{equation*}
From this diagram, since $e$ is an inclusion, we deduce that $\chi_B$ is injective. Since we already know that  $\chi_B$ is surjective, we get that $\chi_B$ is invertible and hence $E=\mathrm{Im}\left( e\right) =\mathrm{Im}\left( p_{I}^{\diamw }\right) .$ This
proves, in view of Lemma \ref{lem:bullet1}, that $B^\diamb\subseteq \sum_{I\in \mathcal{I}_{B}^{f}}\mathrm{Im}\left( p_{I}^{\diamw }\right) $.

Conversely, let $I\in \mathcal{I}_{B}^{f}$ and let us check that $\mathrm{Im}\left( p_{I}^{\diamw }\right)$ belongs to $ \mathcal{S}_{B}^{f}.$ Note that $\varphi
_{2}\left( B/I,B/I\right) :\left( B/I\right) ^{\diamw }\otimes \left(
B/I\right) ^{\diamw }\rightarrow \left( B/I\otimes B/I\right) ^{\diamw }$ is an
injective map between spaces with the same dimension as $\left( B/I\right)
^{\diamw }=\left( B/I\right) ^{\ast }$ and $\left( B/I\otimes B/I\right)
^{\diamw }=\left( B/I\otimes B/I\right) ^{\ast },$ $B/I$ being finite-dimensional. As a consequence $\varphi _{2}\left( B/I,B/I\right) $ is
invertible. Thus we can define a unique $\Delta _{\left( B/I\right) ^{{\diamw
}}}$ such that the following diagram commutes\begin{equation*}
\xymatrix{\left( \frac{B}{I}\right) ^{\diamw } \ar[d]_{\Delta_{\left( \frac{B}{I}\right) ^{\diamw } }} \ar[rrd]^{\left(m_{\frac{B}{I}}\right)^{\diamw}} && \\
\left( \frac{B}{I}\right) ^{{\diamw }}\otimes \left( \frac{B}{I}\right)
^{{\diamw }} \ar[rr]^-{\varphi _{2}\left( \frac{B}{I},\frac{B}{I}\right)}&& \left( \frac{B}{I}\otimes \frac{B}{I}\right) ^{{\diamw }}
 }
\end{equation*}
By using the definition of $\Delta _{\left( B/I\right) ^{{\diamw }}}$ and the naturality of $\varphi_2$, one obtains that
\begin{equation*}
\begin{array}{ccccc}
\xymatrix{\left( \frac{B}{I}\right) ^{{\diamw }}  \ar[rr]^{\left( p_{I}\right)
^{{\diamw }}} \ar[d]_{\Delta _{\left( B/I\right) ^{{\diamw }}}} && B^{\diamw}  \ar[rr]^{\left( m_{B}\right)^{{\diamw }}} &&  \left( B\otimes B\right) ^{{\diamw }} \ar[d]^{1_{ (B\otimes B) ^\diamw }}\\
\left( \frac{B}{I}\right) ^{{\diamw }}\otimes \left( \frac{B}{I}\right)
^{{\diamw }} \ar[rr]^{\left( p_{I}\right) ^{{\diamw }}\otimes \left(
p_{I}\right) ^{{\diamw }}} && B^{{\diamw }}\otimes B^{{\diamw
}} \ar[rr]^{\varphi _{2}\left( B,B\right)} && \left(
B\otimes B\right) ^{{\diamw }}
 }
\end{array}
\end{equation*}This means that $\left( \left( B/I\right) ^{{\diamw }},\Delta _{\left( B/I\right) ^{{\diamw }}},p_{I}^{\diamw }\right) $
is a good $G$-graded vector space of $B^{{\diamw }}$and hence $\mathrm{Im}\left( p_{I}^{\diamw }\right) $ is a good $G$-graded subspace of $B^{{\diamw }}$. Thus $\mathrm{Im}\left( p_{I}^{\diamw }\right)$ becomes an object in $\mathcal{S}_{B}^{f}.$
Summing up we proved that $B^\diamb=\sum_{I\in \mathcal{I}_{B}^{f}}\mathrm{Im}\left( p_{I}^{\diamw }\right) .$

In order to arrive at our goal, we now give another description of $\mathrm{Im}\left( p_{I}^{\diamw
}\right) .$ Let $\xi _{X}:X^{\diamw }\rightarrow X^{\ast }$ denote the
canonical injection. Note that, by the commutativity of the following
diagram
\begin{equation*}
\begin{array}{ccc}
\xymatrix{B^{\diamw } \ar[rr]^{i_{I}^{\diamw }} \ar[d]_{\xi _{B}} && I^{\diamw } \ar[d]^{\xi _{I}}\\
B^{\ast } \ar[rr]^{i_{I}^{\ast }} && I^{\ast }}
\end{array}\end{equation*}and the injectivity of $\xi _{I}$, we get the following alternative
description
\begin{equation*}
\mathrm{Im}\left( p_{I}^{\diamw }\right) \overset{\eqref{seq:ex}}=\mathrm{Ker}\left( i_{I}^{\diamw
}\right) =\mathrm{Ker}\left( \xi _{I}\circ i_{I}^{\diamw }\right) =\mathrm{Ker}\left( i_{I}^{\ast }\circ \xi _{B}\right) =\xi _{B}^{-1}\left( \mathrm{Ker}\left( i_{I}^{\ast }\right) \right) .
\end{equation*}Therefore\begin{eqnarray*}
B^\diamb &=&\sum_{I\in \mathcal{I}_{B}^{f}}\mathrm{Im}\left( p_{I}^{\diamw
}\right) =\sum_{I\in \mathcal{I}_{B}^{f}}\xi _{B}^{-1}\left( \mathrm{Ker}\left( i_{I}^{\ast }\right) \right) \overset{(*)}=\xi _{B}^{-1}\left( \sum_{I\in \mathcal{I}_{B}^{f}}\mathrm{Ker}\left( i_{I}^{\ast }\right) \right) \\
&\overset{(*)}=&\xi _{B}^{-1}\left( \bigcup\limits_{I\in \mathcal{I}_{B}^{f}}\mathrm{Ker}\left( i_{I}^{\ast }\right) \right) =\left\{ f\in B^{\diamw }\mid \xi
_{B}\left( f\right) \in \bigcup\limits_{I\in \mathcal{I}_{B}^{f}}\mathrm{Ker}\left( i_{I}^{\ast }\right) \right\}
\end{eqnarray*}
where in $(*)$ we are using that $\{ \mathrm{Ker}\left( i_{I}^{\ast }\right)\mid I\in \mathcal{I}_{B}^{f}\}$ is a direct set of subobjects of $B^*$ i.e., given $I,J\in \mathcal{I}_{B}^{f}$, there is $K\in \mathcal{I}_{B}^{f}$ such that $\mathrm{Ker}\left( i_{I}^{\ast }\right)\subseteq\mathrm{Ker}\left( i_{K}^{\ast }\right)\supseteq \mathrm{Ker}\left( i_{J}^{\ast }\right) $, namely $K=I\cap J$, see \cite[Theorem 8.6(4)]{Popescu} .
\begin{invisible}
Note that the map $$\frac{B}{I\cap J}\hookrightarrow \frac{B}{I}\times \frac{B}{ J},b+I\cap J\mapsto (b+I,b+J)$$ is injective and hence $I\cap J$ is finite-codimensional whenever $I$ and $J$ are. Moreover $$f\in \mathrm{Ker}( i_I^\ast )\Leftrightarrow f_{\mid I}=0\Rightarrow f_{\mid I\cap J}=0\Leftrightarrow f\in \mathrm{Ker}( i_{I\cap J}^\ast )$$
so that $\mathrm{Ker}( i_I^\ast )\subseteq \mathrm{Ker}( i_{I\cap J}^\ast )$, Similarly one has $\mathrm{Ker}( i_J^\ast )\subseteq \mathrm{Ker}( i_{I\cap J}^\ast )$.

Let $\{X_i\mid i\in I\}$ be a direct set of subobjects. Then $\sum_{i\in I} X_i=\cup _{i\in I} X_i$. In fact given $x\in \sum_{i\in I} X_i$ there are $i_1,\ldots,i_n\in I$ such that $x=x_{i_1}+\cdots +x_{i_n}$. Since the set is direct there is $k\in I$ such that $X_{i_1},\ldots,X_{i_n}\subset X_{k} $ and hence $x\in X_{i_1}+\cdots +X_{i_n}\subseteq X_k$. Therefore $\sum_{i\in I} X_i\subseteq \cup _{i\in I} X_i$. The other inclusion is always true.
\end{invisible}

In conclusion, noting that $\xi _{B}(f)\in \mathrm{Ker}\left( i_{I}^{\ast }\right)$ if and only if $\xi _{B}(f)$ vanishes on $I$, we get
$B^\diamb=\left\{ f\in B^{\diamw }\mid \xi _{B}\left( f\right) \text{
vanishes on some }I\in \mathcal{I}_{B}^{f}\right\} $.
\end{proof}

In conclusion, we got an explicit analogue of Sweedler's finite dual in $\Vec_G$. More generally, having in mind that $\Vec_{G}$ can be regarded as the category of comodules over the group-algebra $kG$, we expect that
one could carry out computations as in $\Vec_G$ for
the category of comodules over a coquasi-bialgebra.

\end{document}